\documentclass[a4paper,12pt, oneside]{article}

\usepackage{amsmath, amssymb, amsthm, pict2e, centernot}
\usepackage{mathtools, faktor, enumitem}

\usepackage{tikz-cd} %
\usepackage{tikz}
\usetikzlibrary{calc, arrows.meta, decorations.markings, intersections}
\tikzset{
  midarrow/.style args={#1}{
    postaction={decorate},
    decoration={markings, mark=at position #1 with {\arrow{Stealth}}}
  }
}

\usepackage{caption}

\usepackage{listings}

\usepackage[page]{appendix}

\setenumerate[1]{label={(\arabic*)}}

\usepackage[T1]{fontenc}
\usepackage{textcomp}
\usepackage[utf8]{inputenc}
\usepackage{lmodern}
\usepackage{microtype}
\usepackage{dsfont}

\newlength{\alphabet}
\usepackage{geometry}
\usepackage[backend=biber, style=numeric, maxbibnames=99, doi=false]{biblatex}
\AtEveryBibitem{\clearlist{language}\clearfield{note}}
\AtEveryBibitem{%
  \ifentrytype{misc}
    {\clearfield{url}\clearfield{urlyear}}
    {}%
}
\usepackage{hyperref}
\hypersetup{
    colorlinks=true,    %
    urlcolor=gray,      %
    linkcolor=darkgray, %
    citecolor=gray, %
    breaklinks=true,
  }
\usepackage{breakurl}

\DeclareFieldFormat*{title}{\mkbibemph{#1}}
\DeclareFieldFormat*{journaltitle}{#1}
\DeclareFieldFormat*{booktitle}{#1}

\renewbibmacro{in:}{%
  \ifentrytype{article}
    {}
    {\printtext{\bibstring{in}\intitlepunct}}}

\DeclareFieldFormat[article,periodical]{volume}{\mkbibbold{#1}}
\DeclareFieldFormat[article,periodical]{number}{\bibstring{number}~#1}
\renewbibmacro*{journal+issuetitle}{%
  \usebibmacro{journal}%
  \setunit*{\addspace}%
  \iffieldundef{series}
    {}
    {\newunit
     \printfield{series}%
     \setunit{\addspace}}%
  \printfield{volume}%
  \setunit{\addspace}%
  \usebibmacro{issue+date}%
  \setunit{\addcolon\space}%
  \usebibmacro{issue}%
  \setunit{\addcomma\space}%
  \printfield{number}%
  \setunit{\addcomma\space}%
  \printfield{eid}%
  \newunit}

\DeclareFieldFormat[article,periodical]{pages}{#1}

\DeclareFieldFormat{mrnumber}{%
  MR\addcolon\space
  \ifhyperref
    {\href{http://www.ams.org/mathscinet-getitem?mr=#1}{\nolinkurl{#1}}}
    {\nolinkurl{#1}}}

\renewbibmacro*{doi+eprint+url}{%
  \iftoggle{bbx:doi}
    {\printfield{doi}}
    {}%
  \newunit\newblock
  \printfield{mrnumber}%
  \newunit\newblock
  \iftoggle{bbx:eprint}
    {\usebibmacro{eprint}}
    {}%
  \newunit\newblock
  \iftoggle{bbx:url}
    {\usebibmacro{url+urldate}}
    {}}

\AtBeginBibliography{\small\sloppy}

\newcommand*{\Z}{\mathbb{Z}}
\newcommand*{\N}{\mathbb{N}}
\newcommand*{\R}{\mathbb{R}}
\newcommand*{\C}{\mathbb{C}}

\newcommand*{\cP}{\mathcal{P}}
\newcommand*{\cQ}{\mathcal{Q}}

\newcommand*{\cB}{\mathcal{B}}

\newcommand*{\cD}{\mathcal{D}}
\newcommand*{\cM}{\mathcal{M}}
\newcommand*{\cA}{\mathcal{A}}
\newcommand*{\cH}{\mathcal{H}}

\newcommand*{\fs}{\mathfrak{s}}

\newcommand*{\rmU}{\mathrm{U}}

\renewcommand{\det}{\operatorname{det}}
\DeclareMathOperator{\id}{id}

\newcommand*{\uwn}[1][]{\hat{w}_{#1}}

\DeclarePairedDelimiter{\ceil}{\lceil}{\rceil}

\theoremstyle{plain}%
\newtheorem*{theorem*}{Theorem}
\newtheorem{theorem}{Theorem}[section]
\newtheorem{lemma}[theorem]{Lemma}
\newtheorem{proposition}[theorem]{Proposition}

\theoremstyle{definition}
\newtheorem{definition}[theorem]{Definition}

\newtheorem*{ack}{Acknowledgements}
\theoremstyle{remark}
\newtheorem*{remark}{Remark}

\newcommand*{\titlestr}{An obstruction to isomorphism of tensor algebras of multivariable dynamical systems}
\newcommand*{\authorstr}{Boris Bilich}
\title{\titlestr}
\author{\authorstr}
\date{\today}

\begin{document}
\begin{center}
	{\LARGE
		\textbf{\titlestr}
	}

	\bigskip

	{\Large \authorstr\footnote{University of Haifa and Georg-August-Universit\"{a}t G\"{o}ttingen \\
	The author is partially supported by the Bloom PhD scholarship and the DFG Middle-Eastern collaboration project no. 529300231\\


	\textbf{Secondary:} 46L55, 37B20.} }

\end{center}
\begin{abstract}
	In their paper on multivariable dynamics, Davidson and Katsoulis conjectured that two multivariable dynamical systems have isomorphic tensor algebras if and only if they are piecewise conjugate.
	We disprove the conjecture by constructing two piecewise conjugate multivariable dynamical systems with four maps on a two-dimensional space, whose tensor algebras are not isomorphic.
\end{abstract}
\section{Introduction}

A classical discrete-time dynamical system consists of a phase space $X$ and a time evolution map $\sigma \colon X \to X$.
This framework captures deterministic processes where the state of the system evolves predictably under the repeated application of $\sigma$.
The notion of a multivariable dynamical system extends this classical idea by replacing the single map $\sigma$ with a finite tuple of maps $\sigma_1, \sigma_2, \ldots, \sigma_n \colon X \to X$.
This extension allows to model systems where the evolution is non-deterministic, governed by multiple possible actions at each step.
A prominent example occurs when $X$ is a complete metric space, and each map $\sigma_i$ is a contraction, forming an iterated function system~\cite{BarnsleyIFS}.
Such systems generate fractals through the chaos game, where a random iterative application of the maps produces a sequence converging to a limit fractal.

A common strategy in the study of geometric objects is to represent their structure algebraically, allowing tools from algebra to reveal deeper dynamical insights.
For an $n$-variable dynamical system on a compact space $X$, the \emph{covariance algebra} is the universal algebra generated by $C(X)$ and elements $\fs_1,\ldots,\fs_n$, satisfying the covariance relations
\[
	f \cdot \fs_i = \fs_i (f\circ \sigma_i), \quad f \in C(X), \; i=1,\ldots,n.
\]
Its universal operator-algebraic completion, where the generators $(\fs_1,\ldots,\fs_n)$ form a row contraction ($\|\sum_{i=1}^n\fs_i \fs_i^*\| \leq 1$), is called the \emph{tensor algebra}.
Historically, these structures were first studied for single-variable dynamical systems by Peters~\cite{Peters84}, influenced by earlier work of Arveson~\cite{Arveson67} and Arveson–Josephson~\cite{AJ69}.
Later, Davidson and Katsoulis~\cite{DK11} introduced tensor algebras of multivariable dynamical systems, which generalize the single-variable case and can be regarded as special cases of tensor algebras of $C^*$-correspondences introduced by Muhly and Solel~\cite{MS98}.

A natural question arises: given two multivariable dynamical systems, when are their tensor algebras isomorphic?
In the single-variable case, this occurs if and only if the underlying dynamical systems are conjugate.
This result was first proven under additional assumptions by Peters~\cite{Peters84} and Hadwin–Hoover~\cite{HH88} and later, without any assumptions, by Davidson and Katsoulis~\cite{DK8}.

Motivated by their success, Davidson and Katsoulis turned their attention to classifying tensor algebras arising from multivariable dynamical systems.
They noted that conjugacy is often too strong a requirement in this setting, as there exist simple examples of non-conjugate systems that nonetheless have isomorphic tensor algebras.
To address this, they introduced the notion of \emph{piecewise conjugacy} (see Definition~\ref{d:piecewise-conjugacy}), which roughly means that, in a neighborhood of every point, the maps of one system can be rearranged via a permutation to match those of the other system.
Note that the choice of permutation may vary from point to point.

Davidson and Katsoulis showed that piecewise conjugacy is an isomorphism invariant of tensor algebras, making it a natural candidate for classification.
They conjectured the converse, namely, that piecewise conjugacy is a complete invariant for tensor algebra isomorphism~\cite[Conjecture 3.26]{DK11}.
They verified the conjecture for up to three maps or when the underlying space has covering dimension at most one.
In this paper, we demonstrate that their conjecture does not hold in general by constructing an explicit counterexample involving four maps on a two-dimensional space.
\begin{theorem*}[\ref{t:pc-not-iso}]
	There exist two piecewise conjugate $4$-variable dynamical systems on a two-dimensional compact metrizable space such that their tensor and covariance algebras are not algebraically isomorphic.
\end{theorem*}

Our approach relies on disproving an auxiliary conjecture proposed by Davidson and Katsoulis, which concerns the existence of a continuous map from a simplex with vertices indexed by the symmetric group $S_n$ to the unitary group $\mathrm{U}(n)$, subject to certain admissibility conditions (see \cite[Conjecture~3.31]{DK11}).
Davidson and Katsoulis showed that the validity of this ``simplex conjecture'' would imply the piecewise conjugacy conjecture.
We show that the converse is also true by constructing two piecewise conjugate multivariable dynamical systems such that an isomorphism of their tensor algebras would yield an admissible map on the $2$-skeleton of the $S_n$-simplex.
Therefore, to find a counterexample to the piecewise conjugacy conjecture, it is enough to show that no such admissible map exists.
We show this by identifying a topological obstruction to the simplex conjecture, demonstrating that such a continuous map does not exist on the $2$-simplex when $n\geq 4$.

The paper is organized as follows.
Sections~\ref{s:Un-invariants} and \ref{s:simplex-conj} constitute the first part of the paper, where we construct a topological obstruction to the simplex conjecture.
Section~\ref{s:Un-invariants} develops homotopy invariants of paths in spaces of unitary matrices.
In Section~\ref{s:simplex-conj}, we introduce admissible maps from a skeleton of the $S_n$-simplex to $\rmU(n)$
The Davidson-Katsoulis simplex conjecture posits that such maps always exist.
We use the invariants from the previous section to establish a complete invariant of homotopy classes of admissible maps on the $1$-skeleton.
Finally, we prove Theorem~\ref{t:no-admissible}, showing that admissible maps on the $1$-skeleton cannot extend to the $2$-skeleton.
The proof reduces to showing that a certain system of linear equations has no solutions, which we verify through computer algebra.

The second part, where we apply our topological obstruction to construct counterexamples to the piecewise conjugacy conjecture, begins with Section~\ref{s:mds-intro}.
After a brief introduction to multivariable dynamical systems and their tensor algebras, we present the crucial notions of piecewise conjugacy and unitary equivalence.
Building upon earlier work by Kakariadis and Katsoulis~\cite{KK14}, Katsoulis and Ramsey~\cite{KR22} proved that tensor algebras of noncommutative multivariable dynamical systems are completely isometrically isomorphic if and only if the systems are \emph{unitarily equivalent after conjugation} (see Definition~\ref{d:unitary-equivalence}).
In Theorem~\ref{t:uc-is-iso}, we strengthen their result for systems without fixed points by showing that in this case, the equivalence holds even for \emph{algebraic} isomorphism.
Building on this, in Section~\ref{s:pe-but-not-ue} we construct two piecewise conjugate $4$-variable dynamical systems such that any unitary equivalence between them would yield an admissible map on the $2$-skeleton.
Since we established the non-existence of such maps, we conclude that the systems are not unitarily equivalent in Theorem~\ref{t:pc-not-ue}.
While this construction does not fully address the piecewise conjugacy conjecture since it does not rule out non-trivial conjugations, we overcome this limitation in Section~\ref{s:rigidification}. There, we introduce a \emph{rigidification} procedure that makes non-trivial conjugations impossible, leading to Theorem~\ref{t:pc-not-iso}, which provides a definitive counterexample to the piecewise conjugacy conjecture.

Finally, we provide an overview of additional related results.
Muhly and Solel developed tensor algebras associated to $C^*$-correspondences and showed that, under an aperiodicity condition, these algebras can be classified up to Morita equivalence by the Morita equivalence class of the underlying correspondence~\cite{MS00}.

In the context of directed graphs, Solel~\cite{Solel04} established that graph tensor algebras completely encode the underlying directed graphs, thereby solving their isomorphism problem in full generality (see also \cite{KK04, Duncan16}).
Additionally, Davidson and Roydor investigated tensor algebras of topological graphs, introducing a notion called local conjugacy, which closely resembles piecewise conjugacy, and proving that algebraic isomorphism implies local conjugacy~\cite{DR11}.
Since multivariable dynamical systems can be viewed as topological graphs, our results also provide an example of two locally conjugate topological graphs whose tensor algebras are not isomorphic.
Recently, Frausino, Ng, and Sims~\cite{FNS23} proposed a cohomological obstruction to isomorphism of correspondences associated to locally conjugate topological graphs.
We believe that a combination of their ideas with the topological obstruction presented in this paper could lead to further progress in the classification of operator algebras associated to dynamics.

Ramsey~\cite{Ramsey16} obtained a complete classification of semicrossed product algebras, another nonselfadjoint operator algebraic completion of the covariance algebra, showing that isomorphisms of these algebras correspond precisely to a dynamical equivalence relation slightly stronger than piecewise conjugacy.

In another direction, Cornelissen and Marcolli~\cite{CM13} showed that piecewise conjugacy appears naturally in the theory of hyperbolic groups and spectral geometry, demonstrating that these dynamical equivalences have relevance far beyond operator algebras alone.
Kakariadis and Shalit~\cite{KS19} studied piecewise conjugacy in the context of operator algebras associated with monomial ideals.

Dor-On~\cite{DorOn18} studied weighted partial systems, extending the classification theory and proving that algebraic or bounded isomorphism of their tensor algebras corresponds precisely to suitable generalized forms of piecewise conjugacy.
Dor-On, Eilers, and Geffen~\cite{DEG20} then unified self-adjoint and non-self-adjoint classification approaches by relating tensor algebra isomorphisms to classification hierarchies of their associated $C^*$-algebras.

\begin{ack}
	I am deeply grateful to my advisor Adam Dor-On for introducing me to this problem, providing numerous valuable suggestions, and carefully proofreading the text.
	I thank my advisor Ralf Meyer for his corrections on a draft of this paper.
\end{ack}

\section{Preliminaries and notation}
Throughout the paper, $n\geq 1$ is a fixed integer.
Let $S_n$ be the symmetric group acting on $\bar n = \{1, \ldots, n\}$.
Fix a total order $\leq$ on $S_n$, for example, the lexicographic order.

A partition of $\bar n$ is a collection $\cP = \{B_1, \ldots, B_k\}$ of non-empty disjoint subsets of $\bar n$ called \emph{blocks} such that $\bar n = B_1 \cup \ldots \cup B_k$.
Given a partition $\cP$, we define an equivalence relation $\sim_{\cP}$ as $i \sim_{\cP} j$ if and only if there exists $B\in \cP$ such that $i,j\in B$.
This defines a bijective correspondence between partitions of $\bar n$ and equivalence relations on $\bar n$.
Given two partitions $\cP, \cQ$ of $\bar n$, we write $\cP \preceq \cQ$ if $\cQ$ is coarser than $\cP$ (i.e., each block of $\cP$ is contained in some block of $\cQ$).

For a collection $g_1, \ldots, g_k \in S_n$, we denote by $\cP(g_1, \ldots, g_k)$ the partition of $\bar n$ into orbits of the action of the group generated by $g_1, \ldots, g_k$.
In particular, $\cP(g)$ is the cyclic decomposition of $g$.

We denote by $\Delta^{S_n}$ the $(n!-1)$-simplex with vertices indexed by $S_n$.
An element $x\in \Delta^{S_n}$ is written as $x = \sum_{g\in S_n} x_g g$ with $x_g \geq 0$ for all $g\in S_n$ and $\sum_{g\in S_n} x_g = 1$.
Given a tuple $g_0, \ldots, g_k \in S_n$ with $g_i < g_{i+1}$ for all $0\leq i < k$, we denote by $\Delta^{g_0, \ldots, g_k}$ the corresponding $k$-subsimplex of $\Delta^{S_n}$ spanned by $g_0, \ldots, g_k$.
The order on $S_n$ induces a canonical orientation on each $\Delta^{g_0, \ldots, g_k}$.
The union of all $k$-subsimplices of $\Delta^{S_n}$ is the $k$-skeleton of $\Delta^{S_n}$, denoted by $\Delta^{S_n}_{(k)}$.

For every subsimplex, we define a partition $\cP(\Delta^{g_0, \ldots, g_k}) =  \cP(g_0^{-1}g_1, \ldots, \allowbreak g_0^{-1}g_k)$.
By definition, $\cP(\Delta^{g_0}) = \{\{1\}, \{2\}, \ldots, \{n\}\}$ is the finest partition of $\bar n$.
\begin{lemma}\label{l:subsimplex-subpartition}
	Suppose that $\Delta^{h_0, \ldots, h_k} \subset \Delta^{g_0, \ldots, g_l}$ are two nested subsimplices.
	Then $\cP(\Delta^{h_0, \ldots, h_k}) \preceq \cP(\Delta^{g_0, \ldots, g_l})$.
\end{lemma}
\begin{proof}
	It is enough to prove that the group $H$ generated by $\{h_0^{-1}h_j\}_{1\leq j \leq k}$ is contained in the group $G$ generated by $\{g_0^{-1} g_i\}_{1\leq i \leq l}$.
	Indeed, the condition $\Delta^{h_0, \ldots, h_k} \subset \Delta^{g_0, \ldots, g_l}$ implies that there are indices $0\leq i_0,\ldots, i_k \leq l$ such that $h_j = g_{i_j}$ for all $0\leq j \leq k$.
	Therefore, we have $h_0^{-1}h_j = g_{i_0}^{-1}g_{i_j} = (g_0^{-1} g_{i_0})^{-1} (g_0^{-1}g_{i_j})$, that is, every generator of $H$ is an element of $G$.
	This concludes the proof.
\end{proof}

Let $\C^n$ be the standard $n$-dimensional Hilbert space with orthonormal basis $\{e_1, \ldots, e_n\}$.
We denote the unitary group of $\C^n$ by $\rmU(n)$.
If $F \subset \bar n$, then we denote by $\rmU(F)$ the unitary group of the subspace $\langle e_i \colon i\in F \rangle$.
For a partition $\cP$, we denote by $\rmU(\cP) = \prod_{B \in \cP} \rmU(B)$ the subgroup of $\rmU(n)$ consisting of all unitary matrices with the block-diagonal form induced by $\cP$.
If $\cP \preceq \cQ$, then $\rmU(\cP) \subseteq \rmU(\cQ)$.

For a block $B\in \cP$, we denote by $\pi_B \colon \rmU(\cP) \to \rmU(B)$ the projection onto the corresponding block.
The $B$-minor map $\det_B = \det\circ \pi_B \colon \rmU(\cP) \to \rmU(1)$ will play a key role in constructing homotopy invariants of paths in $\rmU(\cP)$.

Let $U_g\in \rmU(n)$ denote the permutation matrix corresponding to $g\in S_n$, defined by $U_g e_i = e_{g(i)}$. %
We have $U_g\in \rmU(\cP)$ if and only if $g(B) = B$ for all $B\in \cP$.

\section{Homotopy invariants of paths in spaces of unitary matrices}\label{s:Un-invariants}
In this section, we study paths in spaces $\rmU(\cP)$ up to homotopy.
By a homotopy of paths, we always mean an endpoint-preserving homotopy.
If $X$ is a topological space and $x,y\in X$, we denote by $\pi_1(x, y \mid X)$ the set of homotopy classes of paths in $X$ from $x$ to $y$.
If $X$ is clear from the context, we omit it from the notation.

We denote by $\gamma \ast \gamma'$ the concatenation of two paths, provided the endpoint of $\gamma$ coincides with the initial point of $\gamma'$.
For a path $\gamma$,  we denote the inverse path $t \mapsto \gamma(1-t)$ by $\gamma^{-1}$.
Additionally, if paths take values in a topological group, we define their \emph{pointwise product} $(\gamma \cdot \gamma')(t) = \gamma(t)\gamma'(t)$.
\subsection{Homotopy of paths in \texorpdfstring{$\rmU(1)$}{U(1)}}
In this section, we introduce the main building blocks: homotopy invariants of paths in the smallest unitary group $\rmU(1)$.
Our results rely on standard concepts from algebraic topology, readily available in any textbook on the subject (for example, see \cite[Chapter 3]{Fulton95}).
To suit our specific needs, we adapt notation and provide a concise, self-contained exposition.

From a topological point of view, the unitary group $\rmU(1)$ is the circle $S^1$.
The universal covering map $\sigma\colon \R \to \rmU(1)$ is given by $\sigma(x) = e^{2\pi i x}$.
By the path lifting property, for any path $\gamma\colon [0,1] \to \rmU(1)$, there exists a lift $\hat{\gamma}\colon [0,1] \to \R$ such that $\sigma\circ \hat{\gamma} = \gamma$.
The lift is unique up to an integer shift, i.e., if $\hat{\gamma}_1$ is another lift of $\gamma$, then there is $m\in \Z$ such that $\hat{\gamma}_1(t) = \hat{\gamma}(t) + m$ for all $t\in [0,1]$.
\begin{definition}\label{d:path-invariants}

	For a continuous path $\gamma\colon [0,1] \to \rmU(1)$, we define the following invariants.
	\begin{enumerate}
		\item The \emph{winding number} $w(\gamma)$ is the (not necessarily integer) number of times the path winds counter-clockwise.

		      Precisely, it equals $w(\gamma) = \hat{\gamma}(1) - \hat{\gamma}(0)\in \R$ and it does not depend on the choice of the lift $\hat\gamma$.

		\item The \emph{upper winding number} $\uwn(\gamma) = \ceil*{w(\gamma)} \in \Z$.
		\item The \emph{defect} $\delta(\gamma) = \uwn(\gamma) - w(\gamma) \in [0,1)$ is the fraction of a full counter-clockwise turn that the path $\gamma$ lacks to become a loop.
	\end{enumerate}
\end{definition}

Let $\gamma, \gamma'\colon [0,1]\to \rmU(1)$ be continuous paths.
The winding number satisfies the following additivity properties:
\begin{enumerate}
	\item $w(\gamma \ast \gamma') = w(\gamma) + w(\gamma')$ whenever the concatenation $\gamma \ast \gamma'$ is well-defined (i.e., $\gamma(1) = \gamma'(0)$).
	\item $w(\gamma^{-1}) = -w(\gamma)$.
	\item $w(\gamma\cdot \gamma') = w(\gamma) + w(\gamma')$.
\end{enumerate}

The winding number is a more standard homotopy invariant of paths.
However, it is generally not an integer, except when the path is a loop.
We need its integer-valued counterpart, the upper winding number, in order to perform computer experiments.
The next result sums up the key properties of these invariants.
\begin{proposition}\label{p:U1-invariants}
	Let $x_0, x_1 \in \rmU(1)$ be arbitrary points.
	Then, the following hold.
	\begin{enumerate}
		\item The map $[\gamma] \mapsto w(\gamma)$ is a complete invariant of the homotopy class $[\gamma] \in \pi_1(x_0, x_1)$. \label{p:U1-invariants:w}
		      In particular, when $x_0 = x_1$, a loop $\gamma$ is contractible if and only if $w(\gamma) = 0$.
		\item The map $[\gamma] \mapsto \uwn(\gamma)$ defines a bijection between $\Z$ and $\pi_1(x_0, x_1)$.\label{p:U1-invariants:uw}
		\item The defect $\delta(\gamma)$ depends only on the endpoints $x_0$ and $x_1$, and not on the specific path $\gamma$, i.e., $\delta(\gamma)$ is constant on $\pi_1(x_0, x_1)$. \label{p:U1-invariants:delta}
	\end{enumerate}
\end{proposition}

\begin{proof}
	\begin{itemize}
		\item[\ref{p:U1-invariants:w}] \cite[Problem 3.14]{Fulton95}.

		\item[\ref{p:U1-invariants:delta}] If $\gamma, \tau$ are paths from $x_0$ to $x_1$, then $\gamma \ast \tau^{-1}$ is a loop.
			Therefore, we have $w(\gamma \ast \tau^{-1}) = w(\gamma) - w(\tau) \in \Z$.
			This implies that $w(\gamma)$ and $w(\tau)$ have the same fractional part, i.e., $\delta(\gamma) = \delta(\tau)$.

		\item[\ref{p:U1-invariants:uw}] We have $\uwn(\gamma) = w(\gamma) + \delta(\gamma)$, where $w(\gamma)$ is a complete invariant of the homotopy class and $\delta(\gamma)$ is constant on $\pi_1(x_0, x_1)$.
			Therefore, the assignment $[\gamma] \mapsto \uwn(\gamma)$ is an injective function $\pi_1(x_0, x_1) \to \Z$.
			On the other hand, consider a simple loop $\lambda(t) = e^{2\pi t}x_0$.
			Then, $\uwn(\lambda^m \ast \gamma) = \uwn(\gamma) + m$ for all $m\in \Z$.
			We conclude that any integer $m$ can be realized as $\uwn(\lambda^m \ast \gamma)$ for some path $\gamma$ and, therefore, the map $\uwn\colon \pi_1(x_0, x_1) \to \Z$ is surjective. \qedhere
	\end{itemize}
\end{proof}

\subsection{Homotopy of paths in \texorpdfstring{\(\rmU(\cP)\)}{U(P)}}
The determinant map $\det\colon \rmU(n) \to \rmU(1)$ is a fiber bundle with fiber $\mathrm{SU}(n)$.
By \cite[Proposition~13.11]{Hall15}, $\operatorname{SU}(n)$ is  simply connected.
It follows from the long exact sequence of homotopy groups that the map $\det_* \colon \pi_1(\rmU(n)) \to \pi_1(\rmU(1))$ is an isomorphism.

This implies that two paths $\gamma, \gamma' \colon [0,1]\to \rmU(n)$ with equal endpoints are homotopic if and only if the paths $\det \circ \gamma, \det \circ \gamma'$ are homotopic, if and only if $w(\det \circ \gamma) = w(\det \circ \gamma')$.
We generalize this to $\rmU(\cP)$ as follows.

Let $\cP$ be a partition of $\bar n$ and fix a pair of points $x,y\in \rmU(\cP)$.
Then, two paths $\gamma, \gamma'\colon [0,1]\to \rmU(\cP)$ from $x$ to $y$ are homotopic if and only if the paths $\pi_B\circ \gamma$ and $\pi_B \circ \gamma'$ are homotopic for all $B\in \cP$, where $\pi_B\colon \rmU(\cP) \to \rmU(B)$ is the projection onto the corresponding block.
Consequently, $\gamma$ and $\gamma'$ are homotopic if and only if their images under the $B$-minor map $\gamma_B = \det_B \circ \gamma$ and $\gamma_B' = \det_B \circ \gamma'$ are homotopic for all $B\in \cP$.

For $B\in \cP$, let $w_B(\gamma) = w(\gamma_B)$ and $\uwn[B](\gamma) = \uwn(\gamma_B)$.
Consider invariants $w_{\cP}(\gamma) = (w_B(\gamma))_{B\in \cP} \in \R^{\cP}$ and $\uwn[\cP](\gamma) = (\uwn[B](\gamma))_{B\in \cP} \in \Z^{\cP}$.
\begin{lemma}\label{l:paths-in-UP}
	The invariant $\uwn[\cP]$ defines a bijection between $\pi_1(x, y\mid \rmU(\cP))$ and $\Z^{\cP}$.
	Furthermore, a loop $\gamma$ is contractible in $\rmU(\cP)$ if and only if $w_{\cP}(\gamma) = \uwn[\cP](\gamma) = (0)_{B \in \cP}$.
\end{lemma}
\begin{proof}
	Since $\rmU(\cP) = \prod_{B\in \cP} \rmU(B)$, the assignment $[\gamma] \mapsto ([\pi_B\circ \gamma])_{B\in \cP}$ defines a bijection
	\[
		\pi_1(x, y\mid \rmU(\cP)) \cong \prod_{B\in \cP} \pi_1(\pi_B(x), \pi_B(y) \mid \rmU(B)).
	\]
	Since $\rmU(B)\cong \rmU(\left|B\right|)$, we further have a bijection $\det^* \colon \pi_1(\pi_B(x), \pi_B(y) \mid \rmU(B))\to \pi_1(\det_B(x), \det_B(y)\mid \rmU(1))$.
	Finally, by Proposition~\ref{p:U1-invariants}, $\uwn$ defines a bijection $\pi_1(\det_B(x), \det_B(y))\to \Z$.
	Composing this chain of bijections, we conclude that $[\gamma]\mapsto (\uwn(\det\circ \pi_B \circ \gamma))_{B\in \cP} = \uwn[\cP](\gamma)$ is a bijection
	\[
		\pi_1(x, y\mid \rmU(\cP)) \xrightarrow{\sim} \bigoplus_{B\in \cP} \Z = \Z^{\cP}.
	\]

	The second statement follows immediately from the first one.
\end{proof}

\begin{lemma}\label{l:winding-subpart}
	Suppose that $\cP \preceq \cQ$ are partitions and $\gamma$ is a path in $\rmU(\cP)$.
	When $\gamma$ is considered as a path in $\rmU(\cQ)$, we have
	\[
		w_B(\gamma) = \sum_{\substack{B' \in \cP, \\ B' \subseteq B}} w_{B'}(\gamma)
	\]
	for all $B\in \cQ$.
\end{lemma}
\begin{proof}
	Every matrix from $\rmU(\cP)$ when restricted to $\rmU(B)$ for  $B\in \cQ$ has a block diagonal form with sub-blocks $B'\in \cP$ satisfying $B'\subseteq B$.
	As the determinant of a block-diagonal matrix equals the product of its block determinants, we get
	\[
		\gamma_B = \prod_{\substack{B' \in \cP, \\ B' \subseteq B}} \gamma_{B'}.
	\]
	The statement follows from additivity of the winding number.
\end{proof}

\begin{lemma}\label{l:cdot-W}
	Suppose that $\gamma$ is a path in $\rmU(\cP)$ and $V \in \rmU(\cP)$.
	Consider the path $V\gamma$ given by $[V\gamma](t) = V\cdot \gamma(t)$.
	Then, $w_B(V\gamma) = w_B(\gamma)$ for all $B\in \cP$.
\end{lemma}
\begin{proof}
	The path $V\gamma$ is the pointwise product of the constant path at $V$ and the path $\gamma$.
	The statement follows from the additivity of the winding number.
\end{proof}

\section{Refutation of the simplex conjecture}\label{s:simplex-conj}
\subsection{Admissible maps}\label{s:admissible}
\begin{definition}\label{d:admissible}
	A continuous map $u\colon \Delta^{S_n}_{(k)} \to \rmU(n)$ is called \emph{weakly admissible} if
	\[
		u(\Delta^{g_0, \ldots, g_l}) \subset U_{g_0}\rmU(\cP(\Delta^{g_0, \ldots, g_l}))
	\]
	holds for any subsimplex $\Delta^{g_0, \ldots, g_l} \subset \Delta^{S_n}_{(k)}$.
	If, additionally, $u(g) = U_g$ for all $g\in S_n$, the map $u$ is called \emph{admissible}.
\end{definition}
The simplex conjecture says that admissible maps exist for all $n$ and $k$.
Davidson and Katsoulis established this when either $n\leq 3$ or $k=1$.
In Theorem~\ref{t:no-admissible}, we show that the conjecture fails already for $n=4$ and $k=2$.

\begin{remark}
	We note that in \cite[Conjecture 3.31]{DK11}, the assumption on maps is defined differently.
	We leave it as an exercise to the reader to show that the two definitions are equivalent.
	However, this is not necessary since we use our definition to disprove the piecewise conjugacy conjecture which is shown by Davidson and Katsoulis to be weaker than the simplex conjecture~\cite[Theorem 3.33]{DK11}.
	Therefore, our results imply that the simplex conjecture in its original form is false, even when the equivalence of the two definitions is not established.
\end{remark}

We also introduced the notion of weak admissibility since weakly admissible maps arise naturally as unitary equivalences between certain multivariable dynamical systems (see the proof of Theorem~\ref{t:pc-not-ue}).
It turns out that the existence of weakly admissible maps is equivalent to the existence of admissible maps.
\begin{lemma}\label{l:wadm-iff-adm}
	If there exists a weakly admissible map $u\colon \Delta^{S_n}_{(k)}\to \rmU(n)$, then there exists an admissible map $u'\colon \Delta^{S_n}_{(k)}\to \rmU(n)$.
\end{lemma}
\begin{proof}
	Denote by $\mathrm{DU}(n)$ the group of diagonal unitary matrices.

	Since for every $g\in S_n$, $\cP(\Delta^{g})$ is the partition of $\bar n$ into singletons, we have $\rmU(\cP(\Delta^{g})) = \mathrm{DU}(n)$.
	Therefore, $U_g^*u(g)$ is diagonal and there are diagonal self-adjoint matrices $T_g$ such that $U_g^*u(g) = e^{iT_g}$ for all $g\in S_n$.
	Define a map $v\colon \Delta^{S_n}_{(k)} \to \mathrm{DU}(n)$ by $v(x) = e^{-i\sum_{g\in S_n} x_g T_g}$ for $x = \sum_{g\in S_n} x_g g \in \Delta^{S_n}_{(k)}$.
	We claim that $u'= u\cdot v$ is admissible.

	Indeed, we have $u'(g) = U_g e^{iT_g} e^{-iT_g} = U_g$ for all $g\in S_n$.
	Furthermore, we have
	\[
		u'(\Delta^{g_0, \ldots, g_l}) \subset U_{g_0} \rmU(\cP(\Delta^{g_0, \ldots, g_l})) \mathrm{DU}(n) = U_{g_0}\rmU(\cP(\Delta^{g_0, \ldots, g_l}))
	\]
	since multiplication by a diagonal matrix does not change the block structure.
	Therefore, $u'$ is admissible and we demonstrated the existence of an admissible map.
\end{proof}
\subsection{Homotopy of admissible maps on the 1-skeleton}\label{s:admissible-1-skel}
We now classify admissible maps on the $1$-skeleton of $\Delta^{S_n}$ up to homotopy.
We say that two admissible maps are homotopic if there is a homotopy between them that is admissible at every point.
Since the admissibility condition fixes the values of the maps at the vertices, a homotopy between admissible maps is constant on the vertices.
Therefore, every homotopy is determined by a collection of homotopies of paths determined by edges.
This is where the invariants constructed so far will come into play.

Let $u\colon \Delta^{S_n}_{(1)} \to \rmU(n)$ be an admissible map.
Let $\epsilon^{g_0, g_1}\colon [0,1] \to \Delta^{g_0, g_1}$ be the edge parametrizations given by
\begin{equation}\label{e:edge-parametrization}
	\epsilon^{g_0, g_1}(t) = tg_0 + (1-t)g_1.
\end{equation}
Define continuous paths $\gamma^{g_0, g_1} = U_{g_0}^* (u\circ \epsilon^{g_0, g_1})$ connecting $\mathds 1$ and $U_{g_0}^*U_{g_1}$.
From admissibility it follows that $\gamma^{g_0, g_1}$ is a path in $\rmU(\cP(\Delta^{g_0, g_1}))$ for all edges $\Delta^{g_0, g_1}$ of $\Delta^{S_n}_{(1)}$.

Consider the set
\[
	C = \{ (\Delta^{g_0, g_1}, B) \colon \Delta^{g_0, g_1} \text{ is an edge of } \Delta^{S_n}_{(2)}, B\in \cP(\Delta^{g_0, g_1}) \}.
\]
To an admissible map $u$, we associate the invariants
\begin{align*}
	w_n(u)     & = {\left(w_{B}(\gamma^{g_0, g_1})\right)}_{(\Delta^{g_0, g_1}, B)\in C} \in \R^C,   \\
	\uwn[n](u) & = {\left(\uwn[B](\gamma^{g_0, g_1})\right)}_{(\Delta^{g_0, g_1}, B)\in C} \in \Z^C.
\end{align*}
\begin{proposition}\label{p:admissible-bijection}
	The invariant $\uwn[n]$ defines a bijection between homotopy classes of admissible maps on $\Delta^{S_n}_{(1)}$ and $\Z^C$.
\end{proposition}
\begin{proof}
	As explained above, an admissible map $u \colon \Delta^{S_n}_{(1)} \to \rmU(n)$ defines a collection of paths $\gamma^{g_0, g_1}$ from $\mathds1$ to $U_{g_0^{-1}g_1}$ in $\rmU(\cP(\Delta^{g_0, g_1}))$ for all edges $\Delta^{g_0, g_1}$ of $\Delta^{S_n}$.
	Conversely, given such a collection of paths, we can define an admissible map $u$ by setting $u(tg_0 + (1-t)g_1) = U_{g_0}\gamma^{g_0, g_1}(t)$.

	Analogously, two admissible maps are homotopic if and only if the restriction of the maps to each edge are homotopic as paths.
	Therefore, the set of homotopy classes of admissible maps on $\Delta^{S_n}_{(1)}$ is in bijection with
	\begin{equation}\label{eq:prod-pi}
		\prod_{\Delta^{g_0, g_1}} \pi_1\left(\mathds{1}, U_{g_0^{-1}g_1}\mid \rmU(\cP(\Delta^{g_0, g_1}))\right),
	\end{equation}
	where the product is over all edges $\Delta^{g_0, g_1}$ of $\Delta^{S_n}$.
	The invariant $\uwn[n]$ is just a product of $\uwn[\cP(\Delta^{g_0, g_1})](\gamma^{g_0, g_1})$ for all $\Delta^{g_0, g_1}$ and thus, by Lemma~\ref{l:paths-in-UP}, it defines a bijection of \eqref{eq:prod-pi} with
	\[
		\bigoplus_{\Delta^{g_0, g_1}} \Z^{\cP(\Delta^{g_0, g_1})} = \Z^C.
	\]
\end{proof}

Furthermore, define $\delta_n = (\delta_B(\gamma^{g_0, g_1}))_{(\Delta^{g_0, g_1}, B)\in C} = \uwn[n](\gamma) - w_n(\gamma) \in [0,1)^C$.
This quantity is independent of the choice of $u$, as the endpoints of the paths are fixed by admissibility.
\begin{lemma}\label{l:delta}
	We have
	\[
		\delta_B(\gamma^{g_0, g_1}) = \frac{1 + (-1)^{|B|}}{4} = \begin{cases}
			\frac12 & \text{if } \left|B\right| \text{ is even}, \\
			0       & \text{if } \left|B\right| \text{ is odd}.
		\end{cases}
	\]
	In particular, $\delta_n \in \{0, \frac12\}^C$.
\end{lemma}
\begin{proof}
	The path $\gamma = \gamma^{g_0, g_1}$ starts at $\mathds1$ and ends at $U_{g_0^{-1}g_1}$.
	We have $\det_B (\mathds1) = 1$ and $\det_B (U_{g_0^{-1}g_1}) = \operatorname{sign} (g_0^{-1}g_1|_B)$.
	By definition, $B\in \cP(\Delta^{g_0, g_1}) = \cP(g_0^{-1}g_1)$ is an orbit of the permutation group generated by $g_0^{-1}g_1$ and, hence, $g_0^{-1}g_1|_B$ is a full cycle.
	A cycle is an even permutation if and only if its size is odd:
	\[
		\det_B (U_{g_0^{-1}g_1}) =
		\begin{cases}
			-1 & \text{if } \left|B\right| \text{ is even}, \\
			1  & \text{if } \left|B\right| \text{ is odd}.
		\end{cases}
	\]
	We conclude that the path $\gamma_B = \det_B \circ \gamma$ in $\rmU(1)$ starts at $1$ and ends at either $1$ if $|B|$ is odd and $-1$ otherwise.
	In the first case, it is a loop and in the second case it lacks exactly half of a full turn to be a loop.
	By Definition~\ref{d:path-invariants}, this is equivalent to the formula for $\delta_B$ from the statement.
\end{proof}

\subsection{The obstruction}\label{s:obstruction}
We finally disprove the simplex conjecture.
The main idea is to use the invariants developed above to show that admissible maps defined on the one-dimensional skeleton cannot be extended to the two-dimensional skeleton.

We say that an admissible map $u\colon \Delta^{S_n}_{(k)} \to \rmU(n)$ is \emph{boundary} if it can be extended to an admissible map on $\Delta^{S_n}_{(k+1)}$.
Obviously, there exists an admissible map on a $k+1$-skeleton if and only if there exists a boundary admissible map on the $k$-skeleton.
We will disprove the existence of boundary admissible maps on the $k=1$-skeleton of $\Delta^{S_n}$ for $n = 4$.

Let $u\colon \Delta^{S_n}_{(1)} \to \rmU(n)$ be an admissible map.
Let $\Delta^{g_0, g_1, g_2} $ be a face (a $2$-simplex) of $\Delta^{S_n}_{(2)}$.
Consider the path $\epsilon^{g_0, g_1, g_2} = \epsilon^{g_0, g_1} * \epsilon^{g_1, g_2} * (\epsilon^{g_0, g_2})^{-1}$, where the edge parametrizations $\epsilon^{g_0, g_1}$ were defined in~\eqref{e:edge-parametrization}. %
It is a loop at $g_0$ parametrizing $\partial \Delta^{g_0, g_1, g_2}$.

Then, $\gamma^{g_0, g_1, g_2} = U_{g_0}^*\cdot  (u\circ \epsilon^{g_0, g_1, g_2} )$ is a loop at $\mathds1$.
We have
\begin{align}\label{eq:gamma-3}
	\gamma^{g_0, g_1, g_2} & = \gamma^{g_0, g_1} \ast (U_{g_0}^*U_{g_1}\gamma^{g_1, g_2}) \ast (U_{g_0}^*U_{g_0}\gamma^{g_0, g_2})^{-1} =\nonumber \\
	                       & = \gamma^{g_0, g_1} \ast (U_{g_0^{-1}g_1}\gamma^{g_1, g_2}) \ast (\gamma^{g_0, g_2})^{-1},
\end{align}
where we used that concatenation and inversion of paths commutes with pointwise multiplication by the constant path $U_{g_0}^*$.
This decomposition is illustrated in Figure~\ref{fig:loop}.
\begin{figure}[htb]
	\centering
	\begin{tikzpicture}[scale=4]

		\coordinate (E) at (0,0);
		\node[left] at (E) {$\mathds 1$};
		\fill (E) circle[radius=0.3pt];

		\coordinate (G02) at (1,0);
		\node[right] at (G02) {$U_{g_0^{-1} g_2}$};
		\fill (G02) circle[radius=0.3pt];

		\coordinate (G01) at ($(E)!0.5!(G02) + (0,{1/2})$);
		\node[above] at (G01) {$U_{g_0^{-1} g_1}$};
		\fill (G01) circle[radius=0.3pt];

		\coordinate (G03) at ($(E)!0.5!(G02) - (0,{1/2})$);
		\node[above right] at (G03) {$U_{g_1^{-1} g_2}$};
		\fill (G03) circle[radius=0.3pt];

		\draw[midarrow=0.5] (E) -- (G02) node[midway, below] {$\gamma^{g_0, g_2}$};
		\draw[midarrow=0.5] (E) -- (G01) node[midway, left] {$\gamma^{g_0, g_1}$};
		\draw[dashed, midarrow=0.5] (E) -- (G03) node[midway, below left] {$\gamma^{g_1, g_2}$};
		\draw[midarrow=0.5] (G01) -- (G02) node[midway, right] {$U_{g_0^{-1}g_1}\gamma^{g_1, g_2}$};

		\node at (barycentric cs:E=1,G01=1,G02=1) {$\circlearrowright \gamma^{g_0, g_1, g_2}$};

	\end{tikzpicture}
	\caption{$\gamma^{g_0, g_1, g_2}$ as a concatenation of paths.}
	\label{fig:loop}
\end{figure}
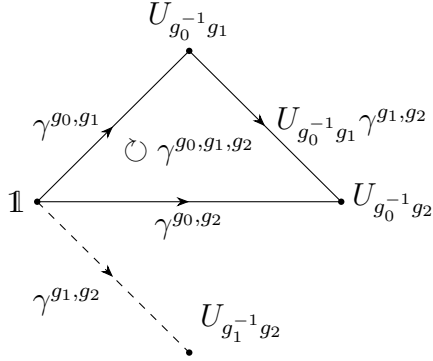

Furthermore, for all $0\leq i < j \leq 2$, $\gamma^{g_i, g_j}$ takes values in $\rmU(\cP(\Delta^{g_i, g_j})) \subset \rmU(\cP(\Delta^{g_0, g_1, g_2}))$.
Since the matrix $U_{g_0^{-1}g_1}$ is also an element of $\rmU(\cP(\Delta^{g_0, g_1, g_2}))$, we conclude that $\gamma^{g_0, g_1, g_2}$ is a loop in $\rmU(\cP(\Delta^{g_0, g_1, g_2}))$.
\begin{lemma}\label{l:boundary-if-contractible}
	An admissible map $u \colon \Delta^{S_n}_{(1)} \to \rmU(n)$ is boundary if and only if the loops $\gamma^{g_0, g_1, g_2}$ are null-homotopic in $\rmU(\cP(\Delta^{g_0, g_1, g_2}))$ for all faces $\Delta^{g_0, g_1, g_2}$ of $\Delta^{S_n}_{(2)}$.
\end{lemma}
\begin{proof}
	Restricting an admissible map on $\Delta^{S_n}_{(2)}$ to each $2$-simplex yields a null-homotopy of the associated loop.
	Conversely, if every loop in $\rmU(\cP(\Delta^{g_0,g_1,g_2}))$ is null-homotopic, then these homotopies glue to form the required extension to the $2$-skeleton, making the map boundary.
\end{proof}

We can compute an obstruction to null-homotopy in the above lemma as a system of linear integer equations.
Let
\[
	R = \{ (\Delta^{g_0, g_1, g_2}, B) \colon \Delta^{g_0, g_1, g_2} \text{ is a face of }\Delta^{S_n}_{(2)}, B\in \cP(\Delta^{g_0, g_1, g_2}) \}
\]
and recall that we have earlier defined
\[
	C = \{ (\Delta^{g_0, g_1}, B) \colon \Delta^{g_0, g_1} \text{ is an edge of } \Delta^{S_n}_{(2)}, B\in \cP(\Delta^{g_0, g_1}) \}.
\]
We define an integer matrix $\cM$ with rows indexed by $R$ and columns indexed by $C$ by
\[
	\cM_{(\Delta^{g_0, g_1, g_2}, B), (\Delta^{h_0, h_1}, B')} = \begin{cases}
		1  & \text{ if } B'\subseteq B \text{ and } (h_0, h_1) = (g_0, g_1), \\
		1  & \text{ if } B'\subseteq B \text{ and } (h_0, h_1) = (g_1, g_2), \\
		-1 & \text{ if } B'\subseteq B \text{ and } (h_0, h_1) = (g_0, g_2), \\
		0  & \text{ otherwise}.
	\end{cases}
\]
We also define a column-vector $\cD$ indexed by $R$ as
\[
	\cD = \cM \delta_n,
\]
where we view $\delta_n$ as a vector in $(\frac12 \Z)^C \subset \mathbb{Q}^C$ (see Lemma~\ref{l:delta}).
In the proof of Proposition~\ref{p:boundary-equation}, we show that $\cD \in \Z^R$.

\begin{proposition}\label{p:boundary-equation}
	An admissible map $u\colon \Delta^{S_n}_{(1)} \to \rmU(n)$ is boundary if and only if $x = \uwn[n](u)$ is a solution to the system of linear integer equations
	\begin{equation}{\label{eq:lin-main}}
		\cM x = \cD.
	\end{equation}
	In particular, there exists a boundary admissible map if and only if there exists an integer solution $x\in \Z^C$ to \eqref{eq:lin-main}.
\end{proposition}
\begin{proof}
	We use the criterion from Lemma~\ref{l:boundary-if-contractible}.
	By \eqref{eq:gamma-3}, we have $\gamma^{g_0, g_1, g_2} = \gamma^{g_0, g_1} \ast (U_{g_0^{-1}g_1}\gamma^{g_1, g_2}) \ast (\gamma^{g_0, g_2})^{-1}$.
	By Lemma~\ref{l:paths-in-UP}, the loop is contractible in $\rmU(\cP(\Delta^{g_0, g_1, g_2}))$ if and only if $w_B(\gamma^{g_0, g_1, g_2}) = w(\det_B \circ \gamma^{g_0, g_1, g_2}) = 0$ for all $B\in \cP(\Delta^{g_0, g_1, g_2})$.
	By additivity of the winding number,
	\[
		w_B(\gamma^{g_0, g_1, g_2}) = w_B(\gamma^{g_0, g_1}) + w_B(U_{g_0^{-1}g_1}\gamma^{g_1, g_2}) - w_B(\gamma^{g_0, g_2}).
	\]
	Furthermore, by Lemma~\ref{l:cdot-W}, we have $w_B(U_{g_0^{-1}g_1}\gamma^{g_1, g_2}) = w_B(\gamma^{g_1, g_2})$ and thus the above expression simplifies to
	\begin{equation}\label{eq:face-winding}
		w_B(\gamma^{g_0, g_1, g_2}) = w_B(\gamma^{g_0, g_1}) + w_B(\gamma^{g_1, g_2}) - w_B(\gamma^{g_0, g_2}).
	\end{equation}
	We claim that the right-hand side is equal to ${\left[\cM w_n(u)\right]}_{(\Delta^{g_0, g_1, g_2}, B)}$.

	Indeed, by Lemma~\ref{l:winding-subpart} applied to $\cP = \cP(\Delta^{g_i, g_j})$ and $\cQ = \cP(\Delta^{g_0, g_1, g_2})$, we have
	\[
		w_B(\gamma^{g_i, g_j}) = \sum \{w_{B'}(\gamma^{g_i, g_j}) \mid B' \in \cP(\Delta^{g_i, g_j}), B'\subseteq B\}
	\]
	for all $0\leq i < j \leq 2$ and $B\in \cP(\Delta^{g_0, g_1, g_2})$.
	Substitute every term on the right-hand side of~\eqref{eq:face-winding} according to the above formula.
	Then, the term $w_{B'}(\gamma^{h_0, h_1})$ appears in  the right-hand side of~\eqref{eq:face-winding} if and only if $B'\subseteq B$ and $(h_0, h_1) = (g_0, g_1), (g_1, g_2)$, or $(g_0, g_2)$, with a coefficient $1$ in the first two cases and $-1$ in the last case.
	This is consistent with the definition of $\cM$:
	\begin{multline*}
		w_B(\gamma^{g_0, g_1, g_2}) \\= \sum_{(\Delta^{h_0, h_1}, B')\in C} \cM_{(\Delta^{g_0, g_1, g_2}, B), (\Delta^{h_0, h_1}, B')} w_{B'}(\gamma^{h_0, h_1}) \\= {\left[\cM w_n(u)\right]}_{(\Delta^{g_0, g_1, g_2}, B)}.
	\end{multline*}

	Let $y$ be the column vector in $\R^R$ given by $y_{(\Delta^{g_0, g_1, g_2}, B)} = w_B(\gamma^{g_0, g_1, g_2})$.
	We conclude that
	\[
		y = \cM w_n(u) = \cM\uwn[n](u) - \cM \delta_n = \cM \uwn[n](u) - \cD.
	\]
	Since $y$ and $\cM \uwn[n](u)$ have integer entries, we get that $\cD$ has integer entries as well.
	From Lemma~\ref{l:boundary-if-contractible} it follows that $u$ is boundary if and only if $w_B(\gamma^{g_0, g_1, g_2}) = 0$ for all $(\Delta^{g_0, g_1, g_2}, B)\in R$, i.e., $y=0$.
	We have shown that this is equivalent to \eqref{eq:lin-main}.
	This concludes the proof of the first part of the proposition and the ``only if'' of the second part.

	Finally, if $x\in \Z^C$ is a solution to \eqref{eq:lin-main}, then there exists an admissible map $u$ such that $\uwn[n](u) = x$ by Proposition~\ref{p:admissible-bijection}.
	We have already shown that this implies that $u$ is boundary.
	We conclude that there exists a boundary admissible map on the $1$-skeleton if and only if there exists an integer solution to \eqref{eq:lin-main}.
\end{proof}
In many cases, it is much easier to determine whether a system of linear equations has a solution over a field rather than over the ring of integers.
Equation~\eqref{eq:lin-main} admits a \emph{rational} solution $x = \delta_n$. %
It is never an integer solution since it has entries in $\{0, \frac12\}$.
Furthermore, it also shows that the equation has a solution modulo any odd prime.
Consequently, the only case where there might fail to be a solution is when considering the equation modulo 2.
The proof of Theorem~\ref{t:no-admissible} is based on this observation.

\begin{theorem}\label{t:no-admissible}
	There is no weakly admissible map on the $2$-skeleton of $\Delta^{S_4}$.
\end{theorem}
\begin{proof}
	By Lemma~\ref{l:wadm-iff-adm}, there is no weakly admissible map on the $2$-skeleton if and only if there is no admissible map on the $2$-skeleton.
	By definition, there is an admissible map on the $2$-skeleton if and only if there is a boundary admissible map on the $1$-skeleton.
	By Proposition~\ref{p:boundary-equation}, this happens if and only if the system \eqref{eq:lin-main} of linear equations has an integer solution.
	A computer-assisted verification\footnote{The code can be found in an ancillary file on the arXiv page of this paper.} shows that for $n=4$, the rank of $\cM$ over $\mathbb{F}_2$ is $462$ while the augmented matrix $(\cM \mid \cD)$ has rank $463$ over $\mathbb{F}_2$. %
	By the Rouch\'e--Capelli theorem~\cite[Theorem~2.38]{SR12}, there is no solution over $\mathbb{F}_2$ and thus no integer solution either.
\end{proof}

Therefore, the simplex conjecture of Davidson and Katsoulis~\cite[Conjecture 3.31]{DK11} is false for $n=4$.
We next apply this result to construct piecewise conjugate multivariable dynamical systems with non-isomorphic tensor algebras.

\section{Multivariable dynamical systems}\label{s:mds-intro}
\subsection{Tensor and covariance algebras}
Having established a topological obstruction to the simplex conjecture, we now apply these results to investigate the tensor algebras arising from multivariable dynamical systems.
We begin by recalling key definitions and notation.
For a detailed exposition, we refer the reader to the original source~\cite{DK11}, where these concepts were first introduced.

A \emph{multivariable dynamical system} $(X, \sigma)$ is a compact space $X$ equipped with a collection of continuous maps $\sigma_i\colon X\to X$ for $1\leq i \leq n$.
To indicate the value of $n$, we also say that $(X,\sigma)$ is an $n$-variable dynamical system.

For a multivariable dynamical system $(X, \sigma)$, the \emph{covariance algebra} $\cA_0(X, \sigma)$ is the universal associative complex algebra generated by a subalgebra $C(X)$ and elements $\fs_1,\ldots, \fs_n$, subject to the relations $f\fs_i = \fs_i(f\circ \sigma_i)$ for all $f\in C(X)$ and $1\leq i \leq n$.

An operator algebra is a closed, but not necessarily self-adjoint, subalgebra of $\cB(\cH)$ for some Hilbert space $\cH$.
A linear operator $\phi\colon \cA_1 \to \cA_2$ between operator algebras is called completely contractive if the induced maps $\phi_m \colon \mathbb{M}_m(\cA_1)\to \mathbb{M}_m(\cA_2)$ are contractive in the operator norm induced from $\mathbb{M}_m(\cB(\cH)) \cong \cB(\cH^m)$ for all $m\geq 1$.
Operator algebras are called completely isometrically isomorphic if there exists a completely contractive isomorphism of algebras with a completely contractive inverse.

Let $\cH$ be a Hilbert space and $\cB(\cH)$ be the algebra of bounded operators on $\cH$.
A representation $\rho \colon \mathcal{A}_0(X,\sigma) \to \mathcal{B}(\mathcal{H})$ is called \emph{row-contractive} if its restriction to $C(X)$ is a $*$-representation, and it satisfies the contraction condition
\[
	\left\|\sum_{i=1}^{n}\rho(\mathfrak{s}_i)\rho(\mathfrak{s}_i)^*\right\| \leq 1.
\]
The \emph{tensor algebra} $\cA(X, \sigma)$ is the operator algebra, universal with respect to row-contractive representations: if $\rho\colon \cA_0(X, \sigma) \to \cB(\cH)$ is a row-contractive representation, then there exists a unique completely contractive representation $\tilde{\rho}\colon \cA(X, \sigma) \to \cB(\cH)$ such that $\tilde{\rho}|_{\cA_0(X, \sigma)} = \rho$.
There exists a faithful row-contractive representation of the covariance algebra, namely the full Fock representation.
This ensures that $\cA_0(X, \sigma)$ is a subalgebra in $\cA(X, \sigma)$.
From the universal property, it follows that $\cA(X, \sigma)$ is unique up to completely isometric isomorphism.

Davidson and Katsoulis~\cite{DK11} posed the problem of identifying necessary and sufficient conditions for isomorphism of tensor algebras \(\mathcal{A}(X, \sigma)\) and \(\mathcal{A}(Y, \tau)\) associated with two multivariable dynamical systems \((X, \sigma)\) and \((Y, \tau)\).
To address this, they introduced several notions of equivalence and conjugacy between multivariable dynamical systems.
Let us recall their definitions.

Consider a multivariable dynamical system \((X, \sigma)\).
If \(\varphi: Y \to X\) is a homeomorphism from another space \(Y\), we can define a new multivariable dynamical system \((Y, \varphi^{-1} \sigma \varphi)\), where \((\varphi^{-1} \sigma \varphi)_i = \varphi^{-1} \circ \sigma_i \circ \varphi\).
In this case, \(\varphi\) is called a \emph{conjugacy} between \((X, \sigma)\) and \((Y, \varphi^{-1} \sigma \varphi)\).

It turns out that conjugacy is too restrictive for classifying tensor algebras.
From the classification of tensor algebras of multivariable dynamical systems on one-dimensional spaces~\cite[Theorem 3.25]{DK11} it follows that there are examples of non-conjugate systems having isomorphic tensor algebras.
Davidson and Katsoulis introduced a weaker notion called \emph{piecewise conjugacy}.

\begin{definition}\label{d:piecewise-conjugacy}
	Two $n$-variable dynamical systems $(X, \sigma)$, $(X, \tau)$ are \emph{piecewise equivalent} if there exists an open cover $\{W_g\}_{g\in S_n}$ of $X$ such that $\tau_i|_{W_g} = \sigma_{g^{-1}(i)}|_{W_g}$ for all $i\in \bar n$ and $g\in S_n$.
	Systems $(X, \sigma)$ and $(Y, \tau)$ are \emph{piecewise conjugate} if there exists a homeomorphism $\varphi\colon X\to Y$ such that $(X, \sigma)$ and $(X, \varphi^{-1}\tau\varphi)$ are piecewise equivalent.
\end{definition}
Piecewise equivalence implies piecewise conjugacy (by taking the identity homeomorphism).
In practice, it is easier to work with piecewise equivalence than with conjugacy since there is no need to account for the choice of homeomorphism between base spaces.

Davidson and Katsoulis proved that piecewise conjugacy is an isomorphism invariant for tensor algebras and conjectured that the invariant is complete~\cite[Conjecture 3.26]{DK11}.
However, this is not the case as we show in Theorem~\ref{t:pc-not-iso}.

To proceed, we introduce a more algebraic notion of equivalence that often coincides with tensor algebra isomorphism.
The definition below is a special case of unitary equivalence of C*-correspondences.
In the context of multivariable dynamical systems, it was studied by Kakariadis and Katsoulis~\cite{KK14}, and by Katsoulis and Ramsey~\cite{KR22}.
\begin{definition}\label{d:unitary-equivalence}
	A \emph{unitary equivalence} between $n$-variable dynamical systems $(X, \sigma)$, $(X, \tau)$ on $X$ is a continuous map $u\colon X\to \rmU(n)$ such that
	\[
		u_{ij}(x) \neq 0 \implies \tau_i(x) =  \sigma_j(x)
	\]
	for all $x\in X$ and $i,j\in \bar n$.
	Systems $(X, \sigma)$ and $(Y, \tau)$ are \emph{unitarily equivalent after conjugation} if there exists a homeomorphism $\varphi\colon X\to Y$ such that the systems $(X, \sigma)$ and $(X, \varphi^{-1}\tau\varphi)$ are unitarily equivalent.
\end{definition}
It is easy to see that unitary equivalence implies piecewise equivalence.
Indeed, since unitaries cannot have zero rows or columns, for every  $x\in X$, there exists a permutation $g\in S_n$ such that $u_{ij}(x)\neq 0$ and thus $u_{i,g^{-1}(i)}$ is non-zero in some neighborhood of $x$.
Therefore, $\tau_i$ equals $\sigma_{g^{-1}(i)}$ in this neighborhood and the systems are piecewise equivalent.
Our approach towards disproving the piecewise conjugacy conjecture basically boils down to showing that the converse is not true.

In a more general context of noncommutative dynamical systems, Katsoulis and Ramsey~\cite{KR22} established that unitary equivalence after conjugation is equivalent to a completely isometric isomorphism of tensor algebras.
In the special case of systems without fixed points, we show in Theorem~\ref{t:uc-is-iso} that even an algebraic isomorphism between tensor algebras implies unitary equivalence after conjugation.

Let $(X, \sigma)$ be a multivariable dynamical system.
For $x\in X$, we define a partition $\cP_\sigma(x)$ by setting $i \sim_{\cP_\sigma(x)} j$ if and only if $\sigma_i(x) = \sigma_j(x)$.
\begin{lemma}\label{l:unitary-conj-block}
	Let $u\colon X\to \rmU(n)$ be a unitary equivalence between $(X, \sigma)$ and  $(X, \tau)$.
	Consider a point $x\in X$ and a permutation $g\in S_n$ satisfying $\tau_i(x) = \sigma_{g^{-1}(i)}(x)$ for all $i\in \bar n$. %
	Then, $u(x) \in U_g \rmU(\cP_\sigma(x))$.
\end{lemma}
\begin{proof}
	It is enough to show that $U_g^*u(x) = U_{g^{-1}} u(x) \in \rmU(\cP_\sigma(x))$.
	For this, we must show that $(U_{g^{-1}} u(x))_{ij} \neq 0$ implies $i \sim_{\cP_\sigma(x)} j$.
	Indeed $(U_{g^{-1}}u(x))_{i j} = u_{g(i) j}(x)\neq 0$, implies $\tau_{g(i)}(x) = \sigma_{j}(x)$ by the definition of unitary equivalence.
	On the other hand, by assumption we have $\tau_{g(i)}(x) = \sigma_{i}(x)$.
	Combining with the previous equality, we get $\sigma_i(x) = \sigma_j(x)$, which precisely means that $i \sim_{\cP_\sigma(x)} j$.
\end{proof}

\subsection{Unitary equivalence for systems without fixed points}\label{s:ue-is-iso}
In this section, we establish an equivalence between unitary equivalence after conjugation and isomorphism of tensor algebras for multivariable dynamical systems without fixed points.
We begin by introducing necessary ring-theoretic concepts.

A ring with marked ideal is a pair $(A, I)$, where $A$ is a unital ring and $I$ is an ideal of $A$.
If $(B, J)$ is some other ring with marked ideal, a ring homomorphism $\alpha\colon A\to B$ is called filtered if $\alpha(I) \subset J$.
In this case, we write $\alpha\colon (A, I)\to (B, J)$.

If $(A, I)$ is a ring with marked ideal, then there is a natural structure of $A/I$-bimodule on $I/I^2$ given by $(a_1 + I)(j + I^2)(a_2 + I) = a_1ja_2 + I^2$.
A filtered homomorphism $\alpha\colon (A,I)\to (B, J)$ induces a homomorphism $\alpha_0\colon A/I\to B/J$ and a linear map $\alpha_1\colon I/I^2\to J/J^2$ which is $A/I$-bilinear with the $A/I$-bimodule structure on $J/J^2$ induced by $\alpha_0$.
If $\alpha$ is an isomorphism such that $\alpha(I) = J$ (so that the inverse of $\alpha$ is also filtered), then $\alpha_0$ and $\alpha_1$ are also isomorphisms.

Let $\cA(X, \sigma)$ be the tensor algebra of a multivariable dynamical system $(X, \sigma)$.
We set $I_{X, \sigma}$ to be the ideal of $\cA(X,\sigma)$ generated by $\fs_1, \ldots, \fs_n$.
We slightly abuse the notation and also denote by $I_{X,\sigma}$ the ideal of $\cA_0(X, \sigma)$ generated by $\fs_1, \ldots, \fs_n$.
This way, we consider rings with marked ideals $(\cA(X, \sigma), I_{X, \sigma})$ and $(\cA_0(X, \sigma), I_{X, \sigma})$.
\begin{lemma}\label{l:associated-bimodule}
	We have $\cA(X, \sigma)/I_{X, \sigma} \cong C(X)$ and $M_\sigma = I_{X, \sigma}/I_{X, \sigma}^2$ is a free right $C(X)$-module with basis $\hat\fs_i = \fs_i + I_{X, \sigma}^2$ for $i=1,\ldots, n$.
	The left $C(X)$-action on $I_{X, \sigma}/I_{X, \sigma}^2$ is given by $f\cdot \hat\fs_i = \hat\fs_i \cdot (f\circ \sigma_i)$ for all $f\in C(X)$.
	The same holds if the tensor algebra is replaced with the covariance algebra $\cA_0(X, \sigma)$.
\end{lemma}
\begin{proof}
	By \cite[Corollary 3.4]{DK11}, every element $a\in \cA(X, \sigma)$ can be written uniquely as
	\[
		a = a_{\emptyset} + \sum_{i=1}^n \fs_i a_i + r(a),
	\]
	where $r(a) \in I_{X, \sigma}^2$ and $a_{\emptyset}, a_i\in C(X)$.
	Furthermore, $a \in I_{X, \sigma}$ if and only if $a_{\emptyset} = 0$.
	From this, we get the desired isomorphism $\cA(X, \sigma)/I_{X, \sigma} \cong C(X)$ and the identification of $I_{X, \sigma}/I_{X, \sigma}^2$ as a free right $C(X)$-module with basis $\hat\fs_i$.
	The formula for the left action follows from the multiplication rule in $\cA(X, \sigma)$.

	The proof for $\cA_0(X, \sigma)$ is analogous.
\end{proof}

A point $x\in X$ is called a fixed point of $(X, \sigma)$ if $\sigma_i(x) = x$ for some $i$.
\begin{lemma}\label{l:auto-filtered}
	Suppose that $(X, \sigma)$ has no fixed points.
	Then, the ideal $I_{X, \sigma}$ equals the intersection of kernels of all characters of $\cA(X, \sigma)$ (resp. $\cA_0(X, \sigma)$).
	In particular, if $(Y, \tau)$ is another multivariable dynamical system, then every isomorphism $\alpha\colon \cA(X, \sigma) \to \cA(Y, \tau)$ or $\alpha\colon \cA_0(X, \sigma) \to \cA_0(Y, \tau)$ is automatically filtered.
\end{lemma}
\begin{proof}
	We will prove the statement for $\cA(X, \sigma)$ since the proof for $\cA_0(X, \sigma)$ is the same.
	Denote by $J$ the intersection of kernels of all characters of $\cA(X, \sigma)\to \C$.
	We first show that $I_{X, \sigma} \subset J$.

	Let $\chi\colon \cA(X, \sigma)\to \C$ be an arbitrary character.
	Then, $\chi|_{C(X)}$ is a character on $C(X)$ and thus corresponds to evaluation at a point $x\in X$.
	Since $x$ is not a fixed point, we can find a function $f\in C(X)$ such that $f(x) = 1$ and $f(\sigma_i(x)) = 0$ for all $1\leq i \leq n$.
	Then,
	\[
		\chi(\fs_i) = f(x) \chi(\fs_i) = \chi(f \fs_i) = \chi(\fs_i(f\circ \sigma_i)) = \chi(\fs_i) f(\sigma_i(x)) = 0.
	\]
	Therefore, $\fs_i\in \ker \chi$ and since $\chi$ was arbitrary, we have $\fs_i\in J$ for all $i$.
	We conclude that $I_{X, \sigma} \subset J$.

	On the other hand, $\cA(X, \sigma)/I_{X, \sigma} \cong C(X)$.
	If $f\in J/I_{X, \sigma} \subset C(X)$, then $f(x) = \operatorname{ev}_x(f) = 0$ for all $x\in X$ and thus $f = 0$.
	Therefore, $J/I_{X, \sigma} = 0$ and we have $J = I_{X, \sigma}$.

	Since the intersection of kernels of characters is an isomorphism invariant, we get that every isomorphism $\alpha\colon \cA(X, \sigma) \to \cA(Y, \tau)$ preserves the ideal $I_{X, \sigma}$.
	We conclude that $\alpha$ is filtered.
\end{proof}

The following theorem is an adaptation of an analogous result proved by Dor-On for weighted partial systems~\cite[Proposition 6.6]{DorOn18}.
\begin{theorem}\label{t:uc-is-iso}
	Suppose that $(X, \sigma)$ and $(Y, \tau)$ are two multivariable dynamical systems without fixed points.
	The following are equivalent.
	\begin{enumerate}
		\item $(X, \sigma)$ and $(Y, \tau)$ are unitarily equivalent after a conjugation.\label{t:uc-is-iso:uc}
		\item The tensor algebras $\cA(X, \sigma)$ and $\cA(Y, \tau)$ are algebraically isomorphic.\label{t:uc-is-iso:iso}
		\item The covariance algebras $\cA_0(X, \sigma)$ and $\cA_0(Y, \tau)$ are isomorphic.\label{t:uc-is-iso:cov}
	\end{enumerate}
\end{theorem}
\begin{proof}
	Let $\varphi\colon X\to Y$ be a homeomorphism and set
	\[
		\tau'=\varphi^{-1}\circ\tau\circ\varphi
	\]
	so that the systems $(X,\sigma)$ and $(X,\tau')$ are unitarily equivalent.
	Since conjugating by $\varphi$ induces an isomorphism
	\[
		\cA_0(Y,\tau)\cong\cA_0(X,\tau'),
	\]
	it suffices to show
	\[
		\cA_0(X,\sigma)\cong\cA_0(X,\tau').
	\]
	Therefore, without loss of generality, we may assume that $X=Y$ and that the systems $(X, \sigma)$ and $(X, \tau)$ are unitarily equivalent.

	Let $u\colon X\to \rmU(n)$ be a unitary equivalence between $(X, \sigma)$ and $(X, \tau)$.
	Define an algebra homomorphism $\alpha_u\colon \cA_0(X, \sigma)\to \cA_0(X, \tau)$ by $\alpha_u(f) = f$ for all $f\in C(X)$ and
	\begin{equation*}%
		\alpha_u(\fs_j^\sigma) = \sum_{i=1}^n \fs_i^\tau u_{ij},
	\end{equation*}
	where $\fs_i^\sigma$ and $\fs_i^\tau$ are the canonical generators of $\cA_0(X, \sigma)$ and $\cA_0(X, \tau)$, respectively.

	We verify that the homomorphism $\alpha_u$ is well-defined.
	For $f \in C(X)$, we compute:
	\begin{equation}\label{eq:alpha-rels}
		\begin{split}
			\alpha_u(f)\alpha_u(\fs^{\sigma}_j) = \sum_{i=1}^n f\fs_i^\tau u_{ij} &= \sum_{i=1}^n \fs_i^\tau (f \circ \tau_i)\cdot u_{ij}, \\
			\alpha_u(\fs^{\sigma}_j)\alpha_u(f \circ \sigma_j) &= \sum_{i=1}^n \fs_i^\tau (f \circ \sigma_j) \cdot u_{ij}.
		\end{split}
	\end{equation}
	By the definition of unitary equivalence, if $u_{ij}(x) \neq 0$ for some $x \in X$, then $\tau_i(x) = \sigma_j(x)$.
	It follows that $(f \circ \tau_i)\cdot u_{ij} = (f \circ \sigma_j)\cdot u_{ij}$ for all $i,j$, and hence both expressions in~\eqref{eq:alpha-rels} are equal.
	Therefore, $\alpha_u$ preserves the covariance relations and is well-defined.
	Moreover, it is an isomorphism, with inverse given by $\alpha_{u^*}$.

	To show that $\ref{t:uc-is-iso:uc} \Rightarrow \ref{t:uc-is-iso:iso}$, we note that if $\rho\colon \cA_0(X, \tau)\to \cB(\cH)$ is a row-contractive representation, then $\rho\circ \alpha_u\colon \cA_0(X, \sigma)\to \cB(\cH)$ is also row-contractive, thanks to the fact that $u$ has unitary values.
	Therefore, the isomorphism $\alpha_u$ extends to a homomorphism of tensor algebras.
	Since the same is true for the inverse $\alpha_u^{-1} = \alpha_{u^*}$, the extended homomorphism is an isomorphism.

	Finally, we show that $\ref{t:uc-is-iso:cov} \Rightarrow \ref{t:uc-is-iso:uc}$ and $\ref{t:uc-is-iso:iso} \Rightarrow \ref{t:uc-is-iso:uc}$.
	By Lemma~\ref{l:auto-filtered}, every isomorphism $\alpha$, as well as its inverse, between tensor or covariance algebras is automatically filtered, and thus induces an isomorphism $\alpha_0\colon C(X) \to C(Y)$.
	If we set $\varphi = (\widehat{\alpha_0})^{-1}\colon X \to Y$, then we may replace $(Y, \tau)$ with $(X, \varphi^{-1}\tau\varphi)$.
	In this way, we may assume without loss of generality that $X=Y$ and $\alpha_0$ is the identity map on $C(X)$.

	Then, $\alpha_1\colon M_\sigma \to M_{\tau}$ is a $C(X)$-bimodule isomorphism as described in Lemma~\ref{l:associated-bimodule}.
	Since both modules are free of rank $n$ as right $C(X)$-modules, there exists an invertible matrix $v\colon X \to \mathrm{GL}(n)$ with entries in $C(X)$ such that
	\[
		\alpha_1(\hat\fs_j^\sigma) = \sum_{i=1}^n \hat\fs_i^{\tau} v_{i,j}, \quad \text{for all } j.
	\]
	Since $\alpha_1$ is also a left $C(X)$-module homomorphism, we have
	\[
		\sum_{i=1}^n \hat\fs_i^{\tau} (f\circ \tau_i) v_{i,j} = \alpha_1(f\hat\fs_j^\sigma) = \alpha_1(\hat\fs_j^\sigma (f\circ \sigma_j))= \sum_{i=1}^n \hat\fs_i^{\tau} (f\circ \sigma_j) v_{i,j}
	\]
	for all $f\in C(X)$.
	Therefore, if $v_{ij}(x) \neq 0$, then $f(\tau_i(x)) = f(\sigma_j(x))$ for all $f\in C(X)$ and thus $\tau_i(x) = \sigma_j(x)$.
	In other words, $v$ satisfies the condition from Definition~\ref{d:unitary-equivalence}.

	The only issue is that $v$ may fail to be unitary.
	To remedy this, we take its polar part $u = v|v|^{-1} \colon X \to \rmU(n)$.
	Since $u$ retains the same block structure as $v$, it defines a unitary equivalence between $(X, \sigma)$ and $(X, \tau)$.
	We conclude that $(X, \sigma)$ is unitarily equivalent to $(X, \tau)$, and hence it is unitarily equivalent after conjugation to $(Y, \tau)$.
\end{proof}

\section{Construction of piecewise but not unitarily equivalent systems}\label{s:pe-but-not-ue}
In this section, we construct two piecewise conjugate multivariable dynamical systems on a space $X$ containing the $2$-skeleton of $\Delta^{S_4}$.
We then show that any unitary equivalence between these systems will necessarily be an admissible map, and thus the systems are not unitarily equivalent.

We denote elements of $\Delta^{S_n}$ as linear combinations
\[
	x = \sum_{g\in S_n} x_g g
\]
satisfying $\sum_{g\in S_n} x_g = 1$.
We have $x\in \Delta^{S_n}_{(k)}$ if and only if $\left| \{g\in S_n \colon x_g \neq 0\} \right| \leq k + 1$.
From now on, we restrict our attention to $\Delta^{S_4}_{(2)}$.

For $g \in S_4$, we define subsets
\[
	V_g = \{ x \in \Delta^{S_4}_{(2)} \colon x_g > \frac14\},
\]
\[
	D_g = \overline{V_g} = \{ x \in \Delta^{S_4}_{(2)} \colon x_g \geq \frac14 \}.
\]
Furthermore, for an ordered tuple $g_0 < g_1 < \ldots < g_l$ of elements of $S_4$, we define
\[
	D_{g_0, \ldots, g_l} = \bigcap_{i=0}^l D_{g_i} =\{ x \in \Delta^{S_n}_{(2)} \colon  x_{g_i} \geq \frac14 \text{ for all } 0\leq i \leq l\}.
\]
Observe that $D_{g_0, \ldots, g_l}$ is empty if $l \geq 3$, i.e., if the simplex $\Delta^{g_0, \ldots, g_l}$ has dimension more than 2 and thus is not a subsimplex of $\Delta^{S_4}_{(2)}$.
Figure~\ref{fig:d-domains} depicts intersections of these sets with a single face.
Note that $D_{g_0}$ and $D_{g_0, g_1}$ are not fully visible in the figure, while $D_{g_0, g_1, g_2}$ fits into a single face.
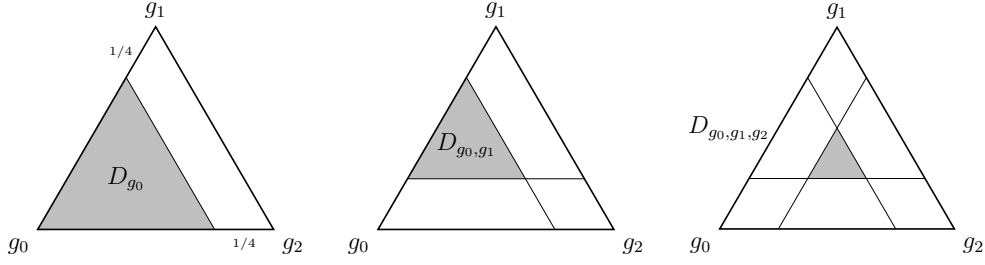
\begin{figure}[htbp]
	\centering
	\begin{minipage}[c]{0.32\linewidth}
		\centering
		\resizebox{\linewidth}{!}{%
			\begin{tikzpicture}[scale=4]
				\coordinate (A) at (0,0);
				\coordinate (C) at (1,0);
				\coordinate (B) at ($(A)!0.5!(C) + (0,{sqrt(3)/2})$);

				\coordinate (M) at ($(B)!0.25!(A)$);
				\coordinate (N) at ($(C)!0.25!(A)$);

				\fill[gray!50] (A) -- (M) -- (N) -- cycle;

				\draw[thick] (A) -- (B) -- (C) -- cycle;

				\draw (M) -- (N);

				\node[below left] at (A) {$g_0$};
				\node[above] at (B) {$g_1$};
				\node[below right] at (C) {$g_2$};

				\node[left] at ($(M)!0.5!(B)$) {\tiny $1/4$};
				\node[below] at ($(N)!0.5!(C)$) {\tiny $1/4$};

				\node at (barycentric cs:A=1,M=1,N=1) {$D_{g_0}$};
			\end{tikzpicture}
		}
	\end{minipage}
	\hfill
	\begin{minipage}[c]{0.32\linewidth}
		\centering
		\resizebox{\linewidth}{!}{%
			\begin{tikzpicture}[scale=4]
				\coordinate (A) at (0,0);
				\coordinate (C) at (1,0);
				\coordinate (B) at ($(A)!0.5!(C) + (0,{sqrt(3)/2})$);

				\coordinate (M) at ($(B)!0.25!(A)$);
				\coordinate (N) at ($(C)!0.25!(A)$);

				\coordinate (M1) at ($(A)!0.25!(B)$);
				\coordinate (N1) at ($(C)!0.25!(B)$);

				\coordinate (M2) at ($(A)!0.25!(C)$);
				\coordinate (N2) at ($(B)!0.25!(C)$);

				\path[name path=MN] (M) -- (N);
				\path[name path=M1N1] (M1) -- (N1);
				\path[name intersections={of=MN and M1N1, by=I01}];

				\fill[gray!50] (M1) -- (M) -- (I01) -- cycle;

				\draw[thick] (A) -- (B) -- (C) -- cycle;

				\draw (M) -- (N);
				\draw (M1) -- (N1);

				\node[below left] at (A) {$g_0$};
				\node[above] at (B) {$g_1$};
				\node[below right] at (C) {$g_2$};

				\node at (barycentric cs:M=1,M1=1,I01=1) {$D_{g_0, g_1}$};
			\end{tikzpicture}
		}
	\end{minipage}
	\hfill
	\begin{minipage}[c]{0.32\linewidth}
		\centering
		\resizebox{\linewidth}{!}{%
			\begin{tikzpicture}[scale=4]
				\coordinate (A) at (0,0);
				\coordinate (C) at (1,0);
				\coordinate (B) at ($(A)!0.5!(C) + (0,{sqrt(3)/2})$);

				\coordinate (M) at ($(B)!0.25!(A)$);
				\coordinate (N) at ($(C)!0.25!(A)$);

				\coordinate (M1) at ($(A)!0.25!(B)$);
				\coordinate (N1) at ($(C)!0.25!(B)$);

				\coordinate (M2) at ($(A)!0.25!(C)$);
				\coordinate (N2) at ($(B)!0.25!(C)$);

				\path[name path=MN] (M) -- (N);
				\path[name path=M1N1] (M1) -- (N1);
				\path[name path=M2N2] (M2) -- (N2);
				\path[name intersections={of=MN and M1N1, by=I01}];
				\path[name intersections={of=MN and M2N2, by=I02}];
				\path[name intersections={of=M1N1 and M2N2, by=I12}];

				\fill[gray!50] (I01) -- (I02) -- (I12) -- cycle;

				\draw[thick] (A) -- (B) -- (C) -- cycle;

				\draw (M) -- (N);
				\draw (M1) -- (N1);
				\draw (M2) -- (N2);

				\node[below left] at (A) {$g_0$};
				\node[above] at (B) {$g_1$};
				\node[below right] at (C) {$g_2$};

				\node[left] at ($(A)!0.5!(B)$) {$D_{g_0, g_1, g_2}$};

			\end{tikzpicture}
		}
	\end{minipage}

	\caption{$D$-domains}
	\label{fig:d-domains}
\end{figure}

\begin{lemma}\label{l:prop-of-covers}
	Let $\Delta^{h_0, \ldots, h_m}$ and $\Delta^{g_0, \ldots, g_l}$ be two subsimplices of $\Delta^{S_4}_{(2)}$.
	\begin{enumerate}
		\item $\{V_g\}_{g\in S_4}$ is an open cover of $\Delta^{S_4}_{(2)}$. \label{l:prop-of-covers:V-cover}
		\item $\Delta^{g_0, \ldots, g_l} \subset \bigcup_{i = 0}^{l} V_{g_i}$. \label{l:prop-of-covers:Delta-V}
		\item If $D_{h_0, \ldots, h_m} \cap \Delta^{g_0, \ldots, g_l} \neq \emptyset$, then $\Delta^{h_0, \ldots, h_m}$ is a subsimplex of $\Delta^{g_0, \ldots, g_l}$. \label{l:prop-of-covers:intersection}
	\end{enumerate}
\end{lemma}
\begin{proof}
	\begin{itemize}
		\item[\ref{l:prop-of-covers:V-cover}] For any $x \in \Delta^{S_4}_{(2)}$, we have $\sum_{g \in S_4} x_g = 1$. Since $x \in \Delta^{S_4}_{(2)}$, at most three of the coefficients $x_g$ are nonzero. Thus, there exists $g \in S_4$ such that $x_g \geq \frac13 > \frac14$. This implies $x \in V_g$, making $\{V_g\}_{g \in S_4}$ a cover.

		\item[\ref{l:prop-of-covers:Delta-V}]
			Consider any $x \in \Delta^{g_0, \ldots, g_l}$. By definition, $x$ is supported only on $\{g_0, \ldots, g_l\}$, and $\sum_{i=0}^l x_{g_i} = 1$. As in the proof of \ref{l:prop-of-covers:V-cover}, there exists some $i$ such that $x_{g_i} > \frac14$. Hence, $x \in V_{g_i}$, showing the inclusion.

		\item[\ref{l:prop-of-covers:intersection}]
			Suppose $x \in D_{h_0, \ldots, h_m} \cap \Delta^{g_0, \ldots, g_l}$. Then $x \in \Delta^{g_0, \ldots, g_l}$, so $x$ is supported only on $\{g_0, \ldots, g_l\}$. Moreover, $x \in D_{h_0, \ldots, h_m}$, implying $x_{h_i} \geq \frac14$ for all $i = 0, \ldots, m$. This ensures $\{h_0, \ldots, h_m\} \subset \{g_0, \ldots, g_l\}$, making $\Delta^{h_0, \ldots, h_m}$ a subsimplex of $\Delta^{g_0, \ldots, g_l}$.
	\end{itemize}
\end{proof}

\begin{lemma}
	Consider a relation on $\Delta^{S_4}_{(2)}\times \{1, 2, 3, 4\}$ given by $(x, i) \sim (y, j)$ if and only if
	\begin{itemize}
		\item $x = y$,
		\item there exists  $D_{g_0, \ldots, g_l}\ni x$ such that $i \sim_{\cP(\Delta^{g_0, \ldots, g_l})} j$.
	\end{itemize}
	It is a closed equivalence relation. %
	In particular, the quotient space  $Z \coloneqq  \Delta^{S_4}_{(2)}\times \{1, 2, 3, 4\} / \sim$ is compact and metrizable.
\end{lemma}
\begin{proof}
	It is immediate that the relation is reflexive and symmetric.
	We are left to show that it is transitive.

	Let $(x, i) \sim (y, j)$ and $(y, j) \sim (z, k)$.
	Since $\sim$ implies equality in the first coordinate, we conclude that $x = y = z$.
	We will show that $(x, i) \sim (x, k)$.

	From $(x, i) \sim (x, j)$, there exists $D_{g_0, \ldots, g_l} \ni x$ such that $i \sim_{\cP(\Delta^{g_0, \ldots, g_l})} j$.
	Similarly, from $(x, j) \sim (x, k)$, there exists $D_{h_0, \ldots, h_m} \ni x$ such that $j \sim_{\cP(\Delta^{h_0, \ldots, h_m})} k$.
	Let $f_0, \ldots, f_s$ be the union of the sets $\{g_0, \ldots, g_l\}$ and $\{h_0, \ldots, h_m\}$ in ascending order (so that $\Delta^{f_0, \ldots, f_s}$ is a correct notation for a simplex).
	Therefore, we have $D_{f_0, \ldots, f_s} = D_{g_0, \ldots, g_l} \cap D_{h_0, \ldots, h_m} \ni x$.
	This, in particular, means that $x$ has at least $s+1$ non-zero coordinates and since we are in $\Delta^{S_4}_{(2)}$, we have $s \leq 2$.
	Consequently, $\Delta^{f_0, \ldots, f_s}$ is a subsimplex of $\Delta^{S_4}_{(2)}$.

	Since $\Delta^{g_0, \ldots, g_l}$ and $\Delta^{h_0, \ldots, h_m}$ are subsimplices of $\Delta^{f_0, \ldots, f_s}$, we have $\cP(\Delta^{g_0, \ldots, g_l}),\allowbreak \cP(\Delta^{h_0, \ldots, h_m})\preceq \cP(\Delta^{f_0, \ldots, f_s})$ by Lemma~\ref{l:subsimplex-subpartition}.
	Therefore, we have $i \sim_{\cP(\Delta^{f_0, \ldots, f_s})} j$ and $j \sim_{\cP(\Delta^{f_0, \ldots, f_s})} k$ and thus $i \sim_{\cP(\Delta^{f_0, \ldots, f_s})} k$.
	It follows that $(x, i) \sim (x, k)$ and the relation is transitive and thus an equivalence relation.

	Since each $D_{g_0, \ldots, g_l}$ is closed and there are only a finite number of them, the equivalence relation is closed. %
	We conclude that the quotient $Z$ is a metrizable compact space (see \cite[Theorems 3.2.11 and 4.2.13]{Engelking}).
\end{proof}

We denote by $[x, i]$ the equivalence class of $(x, i)$ in $Z$.

We define functions $\tilde{\sigma}_i \colon \Delta^{S_4}_{(2)} \to Z$ for $i = 1, 2, 3, 4$ by $\tilde{\sigma}_i(x) = [x, i]$ for all $x\in \Delta^{S_4}_{(2)}$.
Furthermore, we define functions $\tilde{\tau}_i|_{V_g}\colon V_g\to Z$ for $i = 1,2,3,4$ and $g\in S_4$ by
\begin{equation}\label{e:tau}
	\tilde{\tau}_i|_{V_g}(x) = \tilde{\sigma}_{g^{-1}(i)}|_{V_g}(x) = [x, g^{-1}(i)].
\end{equation}

\begin{lemma}
	The functions $\tilde{\tau}_i|_{V_g}$ defined in~\eqref{e:tau} on the elements of the open cover $\{V_g\}_{g\in S_4}$ of $\Delta^{S_4}_{(2)}$ can be glued to a continuous function $\tilde{\tau}_i\colon \Delta^{S_4}_{(2)} \to Z$.
\end{lemma}
\begin{proof}
	We only need to check that the functions agree on the intersections of the open cover.
	Let $x\in V_g\cap V_h \subset D_{g, h}$.
	Then, $\tau_i|_{V_g}(x) = [x, g^{-1}(i)]$ and $\tau_i|_{V_h}(x) =  [x, h^{-1}(i)]$.
	We have $g^{-1}(i) = g^{-1} h h^{-1}(i)$, so $g^{-1}(i)$ and $h^{-1}(i)$ are in the same orbit of the group generated by $g^{-1}h$.
	This means that $g^{-1}(i) \sim_{\cP(\Delta^{g, h})} h^{-1}(i)$ and thus $[x, g^{-1}(i)] = [x, h^{-1}(i)]$.
	We conclude that the functions $\tau_i|_{V_g}$ agree on the intersections $V_g\cap V_h$ and thus can be glued to a continuous function $\tau_i\colon \Delta^{S_4}_{(2)} \to Z$.
\end{proof}

Consider the compact space $X = \Delta^{S_4}_{(2)} \sqcup Z$ and define two $4$-variable dynamical systems $(X, \sigma),~(X, \tau)$ by
\[
	\sigma_i|_{\Delta^{S_4}_{(2)}}  = \tilde\sigma_i,\quad\tau_i|_{\Delta^{S_4}_{(2)}}    = \tilde\tau_i,\quad \sigma_i|_Z = \tau_i|_Z \equiv e\in \Delta^{S_4}_{(2)} \subset X,
\]
where $e$ is the identity element of $S_4$.
We could have chosen any other point of $\Delta^{S_4}_{(2)}$ instead, the only thing that matters is that all the maps have no fixed points.

Recall that given a point $x\in X$ in a multivariable dynamical system, we denote by $\cP_\sigma(x)$ the partition of $\{1, 2, 3, 4\}$ such that $i\sim_{\cP_\sigma(x)} j$ if and only if $\sigma_i(x) = \sigma_j(x)$.
\begin{lemma}\label{l:point-parititon-simplex}
	Suppose that $x\in \Delta^{g_0, \ldots, g_l} \subset \Delta^{S_4}_{(2)} \subset X$.
	Then, $\cP_\sigma(x) \preceq \cP(\Delta^{g_0, \ldots, g_l})$.
\end{lemma}
\begin{proof}
	Suppose that $i\sim_{\cP_\sigma(x)} j$, that is, $[x, i] = \sigma_i(x) = \sigma_j(x) = [x, j]$.
	It is enough to show that $i\sim_{\cP(\Delta^{g_0, \ldots, g_l})} j$.

	By definition, $[x, i] = [x, j]$ implies that there is $\Delta^{h_0, \ldots, h_m}$ such that $x\in D_{h_0, \ldots, h_m}$ and $i\sim_{\cP(\Delta^{h_0, \ldots, h_m})} j$.
	Therefore, $x\in D_{h_0, \ldots, h_m} \cap \Delta^{g_0, \ldots, g_l}$ and thus the intersection is non-empty.
	By Lemma~\ref{l:prop-of-covers}.\ref{l:prop-of-covers:intersection}, $\Delta^{h_0, \ldots, h_m}$ is a subsimplex of $\Delta^{g_0, \ldots, g_l}$.
	By Lemma~\ref{l:subsimplex-subpartition}, we have $\cP(\Delta^{h_0, \ldots, h_m}) \preceq \cP(\Delta^{g_0, \ldots, g_l})$ and, hence, $i\sim_{\cP(\Delta^{g_0, \ldots, g_l})} j$.
\end{proof}

\begin{theorem}\label{t:pc-not-ue}
	The multivariable dynamical systems $(X, \sigma)$ and $(X, \tau)$ defined above are piecewise equivalent but not unitarily equivalent.
\end{theorem}
\begin{proof}
	The systems are piecewise equivalent by construction.
	Specifically, they coincide on $Z$, and $\tau_i|_{V_g} = \sigma_{g^{-1}(i)}|_{V_g}$.
	Thus, the open cover $\{W_g\}_{g \in S_4}$, defined by $W_g = V_g$ for $g \neq e$ and $W_e = V_e \cup Z$, satisfies the conditions of Definition~\ref{d:piecewise-conjugacy}.

	Now, suppose that there exists a unitary equivalence $u: X \to \rmU(4)$ between $(X, \sigma)$ and $(X, \tau)$.
	We claim that the restriction $u|_{\Delta^{S_4}_{(2)}}: \Delta^{S_4}_{(2)} \to \rmU(4)$ is a weakly admissible map.

	To prove this, consider a subsimplex $\Delta^{g_0, \ldots, g_l}$ of $\Delta^{S_4}_{(2)}$, where $l \in \{0, 1, 2\}$, and let $x \in \Delta^{g_0, \ldots, g_l}$ be an arbitrary point.
	By Lemma~\ref{l:prop-of-covers}.\ref{l:prop-of-covers:Delta-V}, there exists an index $0 \leq i \leq l$ such that $x \in \Delta^{g_0, \ldots, g_l} \cap V_{g_i}$.
	From Lemma~\ref{l:unitary-conj-block} and Lemma~\ref{l:point-parititon-simplex}, it follows that $u(x) \in U_{g_i} \rmU(\cP(\Delta^{g_0, \ldots, g_l}))$.

	Therefore, $U_{g_0}^* u(x) \in U_{g_0}^* U_{g_i} \rmU(\cP(\Delta^{g_0, \ldots, g_l}))$.
	Since $U_{g_0}^* U_{g_i} = U_{g_0^{-1} g_i} \in \rmU(\cP(\Delta^{g_0, \ldots, g_l}))$, we deduce that
	\[
		U_{g_0}^* u(x) \in \rmU(\cP(\Delta^{g_0, \ldots, g_l})).
	\]

	Since $x$ was arbitrary, it follows that $u(\Delta^{g_0, \ldots, g_l}) \subset U_{g_0}\rmU(\cP(\Delta^{g_0, \ldots, g_l}))$.
	This precisely means that $u|_{\Delta^{S_4}_{(2)}}$ is weakly admissible (Definition~\ref{d:admissible}).

	This contradicts Theorem~\ref{t:no-admissible}, which states that no such weakly admissible map exists.
	Thus, there can be no unitary equivalence between $(X, \sigma)$ and $(X, \tau)$.
	We conclude that the systems are piecewise equivalent but not unitarily equivalent.
\end{proof}

Since the systems $(X, \sigma)$ and $(X, \tau)$ have no fixed points, from the proof of Theorem~\ref{t:pc-not-ue}, we can deduce that there is no isomorphism between $\cA(X, \sigma)$ and $\cA(X, \tau)$ extending the identity on $C(X)$.
However, it is unclear if this suffices to prove that the tensor algebras are not isomorphic.
In the next section, we circumvent this by introducing a rigidification procedure for multivariable dynamical systems, which allows ruling out non-identity homeomorphisms between the base spaces.

\section{Rigidification of non-equivalence to non-conjugacy}\label{s:rigidification}

Let $X$ be a metrizable compact space.
Fix a countable dense set of points $\{p_k\}_{k\in \N}$ of $X$.
Let $\N^* = \N \cup \{\infty\}$ be the one-point compactification of $\N$.
Consider the subspace $X' \subset X\times \N^*$ given by
\[
	X' = \left(X\times \{\infty\}\right) \cup \{(p_k, i)\colon k\in \N, k\leq i \leq 2k\}. %
\]

\begin{lemma}\label{l:prop-of-xprime}
	The space $X'$ is compact and of covering dimension not greater than the dimension of $X$.
	If there are no isolated points in $X$, then the isolated points in $X'$ are exactly the points of the form $(p_k, l)$ for $l \neq \infty$.
\end{lemma}
\begin{proof}
	For compactness, it is enough to show that $X'$ is a closed subset of a compact space $X\times \N^*$.
	Indeed, suppose that $(x, l)\in X\times \N^*\setminus X'$.
	Then, we have $l<\infty$ and the set $(X\setminus \{p_1, \ldots, p_l\})\times \{l\}$ is an open neighborhood of $(x, l)$ disjoint from $X'$.
	Therefore, $X'$ is closed in $X\times \N^*$ and thus compact.

	The dimension of $\N^*$ is $0$, so the dimension of the product $X\times \N^*$ is less than or equal to the dimension of $X$ by \cite[p.~126, Proposition~2.6]{Pears75}.
	Consequently, the dimension of the closed subset $X'\subset X\times \N^*$ is at most the dimension of $X$ either.

	Finally, to prove the last claim about isolated points, let $l\in \N$.
	The subset $X'_l = X\times \{l\} \cap X' \subset X'$ is clopen and finite and thus consists of isolated points.
	Therefore, every point of the form $(p_k, l)$ for $l\neq \infty$ is isolated in $X'$.

	On the other hand, if $(x, \infty)$ is isolated in $X'$, then it is also isolated in $X\times \{\infty\} \subset X'$ and thus $x$ is isolated in $X$.
	Therefore, if $X$ has no isolated points, then the isolated points in $X'$ are precisely the points of the form $(p_k, l)$ for $l \neq \infty$.
\end{proof}

For a multivariable dynamical system $(X, \sigma)$, we define a (typically not continuous) function $I_\sigma\colon X\to \N^*$ by
\[
	I_\sigma(x) = \#\{ y \in X \colon y \text{ is isolated and }  \sigma_i(y) = x \text{ for some } i\in \bar n\}.
\]
In other words, $I_\sigma(x)$ is the number of isolated points in the union of preimages of $x$ under $\sigma_i$'s.
\begin{lemma}\label{l:I-is-invariant}
	If $(X, \sigma)$ and $(X, \tau)$ are unitarily equivalent, then $I_\sigma = I_\tau$.
\end{lemma}
\begin{proof}
	Let $u\colon X\to \rmU(n)$ be a unitary equivalence between $(X, \sigma)$ and $(X, \tau)$.
	Suppose that $y$ is an isolated point such that $\sigma_j(y) = x$.
	Then, there is at least one $i\in \bar n$ such that $u_{ij}(y) \neq 0$.
	This implies that $\tau_i(y) = \sigma_j(y) = x$.
	It follows that every isolated point in the union of preimages of $x$ under $\sigma_i$'s is also in the union of preimages of $x$ under $\tau_i$'s and thus $I_\sigma(x) \leq I_\tau(x)$.  By symmetry, we conclude that $I_\sigma(x) = I_\tau(x)$.
\end{proof}

We define the \emph{rigidification} of a multivariable dynamical system $(X, \sigma)$ to be the system $(\hat{X}, \hat{\sigma})$, where $\hat{X} = X\sqcup X'$, $\hat\sigma_i|_X = \sigma_i$, and $\hat\sigma_i((x, l)) = x \in X$ for all $i \in \bar n$ and $(x, l)\in X'$.
\begin{lemma}\label{l:rigidification}
	Let $(X, \sigma)$ and $(X, \tau)$ be two multivariable dynamical systems such that $X$ has no isolated points.
	Then, if the rigidifications $(\hat{X}, \hat{\sigma})$ and $(\hat{X}, \hat{\tau})$ are unitarily equivalent after a conjugation, then $(X, \sigma)$ and $(X, \tau)$ are unitarily equivalent.
\end{lemma}
\begin{proof}
	Suppose that $\varphi\colon \hat{X}\to \hat{X}$ is a homeomorphism such that the systems $(\hat{X}, \hat\sigma)$ and $(\hat{X}, \varphi^{-1}\hat{\tau}\varphi)$ are unitarily equivalent.
	It is enough to show that $\varphi|_X = \id_X$.
	Indeed, in this case $\tau = \varphi^{-1}\hat{\tau}\varphi|_X$, $\hat\sigma|_X = \sigma$, and the unitary equivalence between $(\hat{X}, \hat\sigma)$ and $(\hat{X}, \varphi^{-1}\hat{\tau}\varphi)$ restricts to a unitary equivalence between $(X, \sigma)$ and $(X, \tau)$.

	By Lemma~\ref{l:prop-of-xprime}, the isolated points in $\hat{X}$ are of the form $(p_k, l)\in X'$ for $k \leq l \leq 2k$.
	We have $\hat\sigma_i((p_k,l)) = \hat\tau_i((p_k, l)) = p_k \in X$ for all $i\in \bar n$ and $k\in \N$.
	Therefore, we have $I_{\hat\sigma}(p_k) = I_{\hat\tau}(p_k) = k + 1$ for all $k \in \N$ and $I_{\hat\sigma}(y) = I_{\hat\tau}(y) = 0$ for all $y\in \hat{X}$ not equal to any $p_k$.
	In particular, $p_k$ is the unique point in $\hat{X}$ satisfying $I_{\hat\tau}(p_k) = k+1$.

	By Lemma~\ref{l:I-is-invariant}, we have $I_{\hat\sigma} = I_{\varphi^{-1}\hat\tau\varphi} = I_{\hat\tau}\circ \varphi$.
	Therefore, we have $k+1 = I_{\hat\sigma}(p_k) = I_{\hat\tau}(\varphi(p_k))$.
	By uniqueness, $\varphi(p_k) = p_k$ for all $k\in \N$.
	Since the set of points $\{p_k\}_{k\in \N}$ is dense in $X$, we have $\varphi|_X = \id_X$.
\end{proof}

\begin{theorem}\label{t:pc-not-iso}
	There exist two piecewise conjugate $4$-variable dynamical systems on a two-dimensional compact metrizable space such that their tensor and covariance algebras are not algebraically isomorphic.
\end{theorem}
\begin{proof}
	Consider the two $4$-variable dynamical systems $(X, \sigma)$ and $(X, \tau)$ defined in Section~\ref{s:pe-but-not-ue}.
	By construction, $X$ has a finite closed covering by 2-dimensional simplices.
	Therefore, the dimension of $X$ is $2$ by \cite[p.~125, Theorem~2.5]{Pears75}.
	Consequently, the dimension of the rigidifications $(\hat{X}, \hat{\sigma})$ and $(\hat{X}, \hat{\tau})$ is also $2$ by Lemma~\ref{l:prop-of-xprime}.

	We claim that $(\hat{X}, \hat{\sigma})$ has no fixed points.
	Indeed, for all $1\leq i \leq 4$, we have $\hat\sigma_i(X') \subset X$, so there are no fixed points in $X'$.
	Analogously, $X$ consists of two components, $Z$ and $\Delta^{S_4}_{(2)}$, and the maps $\hat\sigma_i$ act from each of these components to the other.
	So, there are no fixed points in $X$ either.

	By Theorem~\ref{t:pc-not-ue}, the systems $(X, \sigma)$ and $(X, \tau)$ are not unitarily equivalent.
	Therefore, by Lemma~\ref{l:rigidification}, the rigidifications $(\hat{X}, \hat{\sigma})$ and $(\hat{X}, \hat{\tau})$ are not unitarily equivalent after a conjugation.
	By Theorem~\ref{t:uc-is-iso}, this implies that both the tensor and covariance algebras of $(\hat{X}, \hat{\sigma})$ and $(\hat{X}, \hat{\tau})$ are not algebraically isomorphic.
\end{proof}
Theorem~\ref{t:pc-not-iso} provides a counterexample to the piecewise conjugacy conjecture by Davidson and Katsoulis~\cite[Conjecture 3.26]{DK11}.

\printbibliography
\end{document}